\documentclass[11pt]{amsart}

\usepackage[mathscr]{eucal}
\usepackage{amsmath,amssymb,amsfonts,amsthm,enumerate}

\usepackage{hyperref}

\newtheorem{theorem}{Theorem}[section]
\newtheorem{lemma}[theorem]{Lemma}
\newtheorem{corollary}[theorem]{Corollary}
\newtheorem{proposition}[theorem]{Proposition}
\newtheorem{conjecture}[theorem]{Conjecture}

\theoremstyle{definition}
\newtheorem{example}{Example}

\newcommand{\Z}{\mathbb Z}
\newcommand{\R}{\mathbb R}

\newcommand{\C}{\mathbb C}

\DeclareMathOperator{\ord}{ord}

\newcommand{\la}{\langle}
\newcommand{\ra}{\rangle}
\newcommand{\be}{\begin{equation}}
\newcommand{\ee}{\end{equation}}
\newcommand{\und}{\;\mbox{ and }\;}
\newcommand{\nn}{\nonumber}
\newcommand{\ber}{\begin{eqnarray}}
\newcommand{\eer}{\end{eqnarray}}
\newcommand{\Sum}[2]{\underset{#1}{\overset{#2}{\sum}}}
\newcommand{\Summ}[1]{\underset{#1}{\sum}}

\DeclareSymbolFont{goo}{OMS}{cmsy}{b}{n}
\DeclareMathSymbol{\gooT}{\mathalpha}{goo}{"1}

\begin{document}

\title[The $3k-4$ Theorem modulo a Prime]{The $3k-4$ Theorem modulo a Prime:
 High Density for $A+B$}

\author{David J. Grynkiewicz}
\address{Department of Mathematical Sciences\\ University of Memphis\\ Memphis, TN 38152\\
USA}
\email{diambri@hotmail.com}


\begin{abstract}
The $3k-4$ Theorem for $\mathbb Z$ asserts that, if $A,\,B\subseteq \mathbb Z$ are finite, nonempty subsets with $|A|\geq |B|$ and $|A+B|=|A|+|B|+r< |A|+2|B|-3$, then there exist arithmetic progressions $P_A$ and $P_B$  of common difference such that $X\subseteq P_X$ with $|P_X|\leq |X|+r+1$ for all $X\in \{A,B\}$. There is much partial progress extending this result to $\Z/p\Z$ with $p\geq 2$ prime. Here, among other results, we begin by showing that, if $A,\,B\subseteq G=\mathbb Z/p\mathbb Z$ are nonempty subsets with $|A|\geq |B|$, \ $A+B\neq G$, \ $|A+B|=|A|+|B|+r\leq  |A|+1.0527|B|-3$,
and $|A+B|\leq |A|+|B|-9(r+3)$, then there exist arithmetic progressions $P_A$, \ $P_B$ and $P_C$  of common difference such that $X\subseteq P_X$ with $|P_X|\leq |X|+r+1$ for all $X\in \{A,B,C\}$, where $C=-\,G\setminus (A+B)$. This gives a rare high density version of the $3k-4$ Theorem for general sumsets $A+B$ and is  the first instance with tangible (rather than effectively existential) values for the constants for general sumsets $A+B$ without also imposing additional constraints on the relative sizes of $|A|$ and $|B|$. The ideal conjectured density restriction under which a version of the $3k-4$ Theorem modulo $p$ is expected is $|A+B|\leq p-(r+3)$. In part by utilizing the above result as well as several other recent advances, we are  able to extend  methods of Serra and Z\'emor to give a version valid under this ideal density constraint. We show that,  if $A,\,B\subseteq G=\mathbb Z/p\mathbb Z$ are nonempty subsets with $|A|\geq |B|$, \ $A+B\neq G$, \ $|A+B|=|A|+|B|+r\leq  |A|+1.01|B|-3$,
and $|A+B|\leq |A|+|B|-(r+3)$, then there exist arithmetic progressions $P_A$, \ $P_B$ and $P_C$  of common difference such that $X\subseteq P_X$ with $|P_X|\leq |X|+r+1$ for all $X\in \{A,B,C\}$, where   $C=-\,G\setminus (A+B)$. This notably improves upon the original result of Serra and Z\'emor \cite{serra-3k-4-plunnecke}, who treated the case $A+A$, required $p$ be sufficiently large, and needed the  much more restrictive small doubling hypothesis $|A+A|\leq  |A|+1.0001|A|$.
\end{abstract}

\maketitle

\section{introduction}

Let $G$ be an abelian group. For  subsets $A,\,B\subseteq G$, their sumset and difference set are
$$A+B=\{a+b:\:a\in A,\,b\in B\}\;\und\; A-B=\{a-b:\; a\in A,\,b\in B\}.$$ When $A=B$, we write $2A=A+A$. When $B=\{x\}$ is a single element, we write $\{x\}+A=x+A=A+x$ for the translate of $A$ by $x$. We let $$ A^{\mathsf c}=G\setminus A$$ denote the complement of $A$. For a nonempty subset $A\subseteq \Z$, we let $\gcd(A)$ denote the greatest common divisor of all elements in $A$.  We often use $[m,n]=\{x\in \Z:\; m\leq x\leq n\}$ to denote the discrete interval between $m$ and $n$, for any $m,n\in \R$, with context making it clear when $[m,n]$ should instead be considered as an interval of real numbers or an interval modulo $p$ (and thus as a subset of $\Z/p\Z$).

The study of sumsets is a core topic in Additive Combinatorics.
The following is a  modern version of  the classical Freiman $3k-4$ Theorem. It says that an integer sumset $A+B\subseteq \Z$ with very small sumset must have $A$ and $B$ contained in short arithmetic progressions with $A+B$ containing a long length arithmetic progression, all with the same common difference. Notable is the fact that all bounds, including those for the progressions $P_A$, $P_B$ and $P_{A+B}$ as well as the hypothesis \eqref{hyp3k4} are tight, and that the result holds true for distinct summands $A\neq B$ (including the commonality of the difference) rather than only treating the case $A=B$. Having either of these desirable  features is rare in this area of Additive Combinatorics, and they are achieved at the cost of imposing a strong small doubling restriction on the sumset $|A+B|$ that, when $A=B$ with $|A|=k$, amounts to $|A+B|\leq 3k-4$. Hence the name.

\begin{theorem}[$3k-4$ Theorem]\label{thm-3k4Z}
Let $A,\,B\subseteq \Z$ be finite and nonempty with  $|A|\geq |B|$   and \be\label{hyp3k4}|A+B|=|A|+|B|+r\leq |A|+2|B|-3-\delta,\ee
where $\delta=1$ if $A=x+B$ for some $x\in \Z$ and otherwise $\delta=0$.
Then there are arithmetic progressions $P_A,\,P_B,\,P_{A+B}\subseteq \Z$  having  common difference $d:=\gcd(A-A+B-B)$ such that
\begin{align*} &X\subseteq P_X\;\und\;|P_X|\leq |A|+r+1,\quad
\mbox{ for all $X\in \{A,B\}$, while}\\
&P_{A+B}\subseteq A+B\;\und\;|P_C|\geq |A|+|B|-1.\end{align*}
\end{theorem}

Without the conclusions regarding $P_{A+B}$ or the refinement involving $\delta$, Theorem \ref{thm-3k4Z} may be derived from an old result of Freiman \cite{Freiman-3k4-distinct} (see the arguments from \cite{3k-4-stanchescu}) with him having treated the case $A=B$ in a separate earlier work \cite{3k-4-Freiman-A=B}. An alternative (and now standard) proof  was given by  Lev and Smelianksy \cite{3k-4-lev-smel}, who derive a variation where $B$ is assumed to be the set with smaller diameter rather than smaller cardinality.  The conclusion regarding $P_{A+B}$ was derived by  Freiman for $A=B$ in \cite{3k-4-Freiman-AP} with the case $A\neq B$ treated in  \cite{3k-4-itziar}.  The modern text \cite{Grynk-book}  collects all these versions together into one source and includes the added refinement involving  $\delta$.
The tightness of all bounds in Theorem \ref{thm-3k4Z} can be seen by considering the following  examples.

\begin{example} Take $$A=\{0,2,\ldots,2(r+1)\}\cup [2r+3,m]\;\und\;B=\{0,2,\ldots,2(r+1)\}\cup [2r+3,n],$$ where $r\geq -1$, \ $m\geq n\geq 2r+3$ and $m\geq 2r+4$. Then $A+B=\{0,2,\ldots,2(r+1)\}\cup [2r+3,m+n]$ with $|A|=m-r\geq r+4$, \ $|A|\geq |B|=n-r\geq r+3$ and $|A+B|=m+n-r=|A|+|B|+r$, while $P_A=[0,m]$, \ $P_B=[0,n]$ and $P_C=[2r+2,m+n]$. This shows all bounds in the conclusion of Theorem \ref{thm-3k4Z} can hold with equality simultaneously.
\end{example}

\begin{example} Let $r\geq 0$ and let   $A=B=P_1\cup P_2$ be the union of two arithmetic progressions with common difference such that $|A|=r+3$ and $\max P_1$ is much smaller than $\min P_2$. Then $A+B=A+A$ with $|A+A|=3|A|-3=2|A|+r$, \  $A=B$ and  $|A|=|B|=r+3$, and the conclusion of Theorem \ref{thm-3k4Z} cannot hold.
\end{example}

\begin{example}Take $B=[0,r+1]$ with $r\geq -1$ and take $A=P_1\cup P_2$ to be a union of two intervals with $|A|\geq \max\{|B|,3\}$ and  $\max P_1$ much smaller than $\min P_2$. Then $|A+B|=|A|+2|B|-2=|A|+|B|+r$ with $|A|\geq |B|=r+2$, and   the conclusion of Theorem \ref{thm-3k4Z} cannot   hold.
\end{example}

Theorem \ref{thm-3k4Z} is valid for $G=\Z$. It is a general principle that, so long as $|A|$ and $|B|$ are extremely small  \cite{lev-rect-threshold} \cite[Chapter 20]{Grynk-book}, then a sumset $A+B$ in a finite abelian group $G$ has the same behaviour as that of an integer sumset. Thus one can expect a version of Theorem \ref{thm-3k4Z} to hold in a finite abelian group $G$ provided the sets involved are not too large. In this paper, we are concerned with such extensions for $G=\Z/p\Z$ with $p$ prime. The question of extending Theorem \ref{thm-3k4Z} to $\Z/p\Z$ has received considerable attention. In order to help orient ourselves, we begin with the following conjecture representing the current best-case scenario that could hold in $\Z/p\Z$. We remark that the definitions of when $\delta_B=1$ or $\delta_C=1$ are still tentative below. The reader who wishes to ignore the complicated definitions of  $\delta_C$ and $\delta_B$  can simply take both these constants equal to the constant $\delta$, whose value is  $1$ when one of the sets $A$, $B$ or $C$ is a translate of another of these sets and otherwise equals $0$.

\begin{conjecture}\label{conj-crit-pair}
Let $G=\Z/p\Z$ with $p\geq 2$  prime, let $A,\, B\subseteq G$ be nonempty subsets with $A+B\neq G$, and set $C=-\,G\setminus (A+B)$. Suppose
  $|A|\geq |B|$   and \be\nn|A+B|=|A|+|B|+r\leq \min\{|A|+2|B|-3-\delta_B,\; p-r-3-\delta_C\},\ee
  where, for $X\in\{A,B,C\}$ with $(X,Y,Z)$ a permutation of $(A,B,C)$, $$\delta_X=\left\{
             \begin{array}{ll}
               1 & \hbox{if $r\geq 0$ and $X$ is a translate of either $Y$ or $Z$} \\
               1 & \hbox{if $r\geq 2$, \ $Y$ and $Z$ are translates of each other, and   $|Y|=|Z|=r+4$} \\
               0 & \hbox{otherwise.}
             \end{array}
           \right.$$   
Then there are arithmetic progressions $P_A,\,P_B,\,P_{C}\subseteq G$ all having  common difference  such that
\begin{align*} &X\subseteq P_X\;\und\; |P_X|\leq |X|+r+1\quad \mbox{ for all $X\in \{A,B,C\}$}.\end{align*}
\end{conjecture}

The reader may notice that the conclusion in Conjecture \ref{conj-crit-pair} has a more symmetric formulation in terms of complementation than that of Theorem \ref{thm-3k4Z}. However, setting $$P_{A+B}=(-P_C)^\mathsf c=-(P_C)^\mathsf c=-\, G\setminus P_C,$$ it follows that $P_{A+B}$ is also an arithmetic progression with the same common difference as $P_C$, and that the conclusion $C\subseteq P_C$ with $|P_C|\leq |C|+r+1$ is equivalent to $P_{A+B}\subseteq A+B$ with $|P_{A+B}|=p-|P_C|\geq p-|C|-r-1=|A+B|-r-1=|A|+|B|-1$, which is what is stated in Theorem \ref{thm-3k4Z}.
To explain the additional density restriction $|A+B|\leq p-r-3-\delta_C$ given in Conjecture \ref{conj-crit-pair}, consider the following examples.

\begin{example}
Fix $r\geq -1$,  let $C=[0,r+1]\subseteq \Z$ and let $B=P_1\cup P_2\subseteq \Z$ be a union of two intervals with $|B|\geq \max\{r+2,3\}$, \  $\min P_1=0$ and  $\min P_2-\max P_1= M$ for some large constant $M\geq r+3$.  Then $|B|\geq|C|$ and $|B+C|=|B|+2|C|-2=|B|+|C|+r$ with  $P_B=[0,\max P_2]$ the minimal arithmetic progression containing $B$ and $|P_B|=|B|+M-1$, which is unbounded with respect to $|B|$ or $r$.  Fix any large prime $p$ and now consider $B$ and $C$ as subsets of $G=\Z/p\Z$, so we consider the original sets modulo $p$. For $p$ sufficiently large, the set $B\subseteq G$  still can  be covered no more effectively by an arithmetic progression than by using the progression $P_B$ considered modulo $p$, meaning the conclusion of Conjecture \ref{conj-crit-pair} fails  by an arbitrarily  large margin for $B$. Let
 $A=-(B+C)^\mathsf c$. Then $|A|\geq |B|$ (for $p$ sufficiently large).
 It is easily noted that $B+(C\cup \{x\})\neq B+C$ for all $x\in  C^\mathsf c$  (for $p$ sufficiently large). Thus Vosper Duality  (stated in Proposition \ref{lemma-dual}) ensures that $A+B=-C^\mathsf c$ with $|A+B|=|A|+|B|+r$, and now $A+B$ is an example for which the conclusions of  Conjecture \ref{conj-crit-pair} fail  with  $|A+B|=|C^\mathsf c|=p-r-2$ and $|A|\geq |B|\geq |C|=r+2$. \end{example}

 \begin{example} Fix $r\geq 0$ and let  $B=P_1\cup P_2\subseteq \Z$ be a union of two intervals with $|B|=r+3$, \  $\min P_1=0$ and  $\min P_2-\max P_1=M$ for some large constant $M$.  Then $ |B+B|=3|B|-3=2|B|+r$ with  $P_B=[0,\max P_2]$ the minimal arithmetic progression containing $B$ and $|P_B|=|B|+M-1$, which is unbounded with respect to $|B|=r+3$.  Fix any large prime $p$ and now consider $B$  as a subset of $G=\Z/p\Z$. For $p$ sufficiently large, the set $B\subseteq G$  still can  be covered no more effectively by an arithmetic progression than by using the progression $P_B$ considered modulo $p$, meaning the conclusion of Conjecture \ref{conj-crit-pair} fails  by an arbitrarily  large margin for $B$. Let $A=-(2B)^\mathsf c$.
 It is easily noted that $B+(B\cup \{x\})\neq B+B$ for all $x\in B^\mathsf c$  (for $p$ sufficiently large). Thus Vosper duality  (stated in Proposition \ref{lemma-dual}) ensures that $A+B=-B^\mathsf c$ with $|A+B|=|A|+|B|+r=|A|+2|B|-3$ and $|A|\geq |B|$ (for $p$ sufficiently large), and now  $A+B$ is a   counter-example to Conjecture \ref{conj-crit-pair} holding with  $|A+B|=|B^\mathsf c|=p-|B|=p-r-3$, $B=C$ and $|A|\geq |B|=|C|=r+3$.
\end{example}

\begin{example} Fix $r\geq 2$ with $p=4r+11$ prime and let  $A=B=[0,r+1]\cup \{r+3,2r+6\}=[0,r+1]\cup \{\frac{p+1}{4},\frac{p+1}{2}\}$ considered modulo $p$ (so as subsets of $\Z/p\Z$).  Then $|A|=r+4$ and $A+B=2A=[0,2r+4]\cup [2r+6,3r+7]\cup \{3r+9\}$ with $|A+A|=3r+8=3|A|-4=2|A|+r=p-r-3$. By Proposition \ref{prop-pablo-eg-works}, the minimal length of an arithmetic progression $P_A$ containing $A$ is $2r+7=|A|+r+3$, meaning $A+B=A+A$ fails to satisfy the conclusion of Conjecture \ref{conj-crit-pair} with $|A+A|=3|A|-4=2|A|+r=p-r-3$, \ $A=B$,  $|A|=|B|=r+4$ and $|C|=r+3$.
\end{example}

Unlike Examples 1--5, Example 6 was  at first missed. Its discovery in \cite{pablo-serra-3k4} led to the  slightly  modified Conjecture \ref{conj-crit-pair} versus earlier versions of the conjecture (e.g., \cite{serra-3k-4-plunnecke}). Since the details showing Example 6 is a counterexample were not given  in \cite{pablo-serra-3k4}, we have included them in Proposition \ref{prop-pablo-eg-works} below.  Worth noting, this construction requires $r\geq 2$, and the failure in the conclusion of Conjecture \ref{conj-crit-pair} is only off by $2$, rather than an unbounded amount as for the other examples.


For the purpose of overviewing previous progress towards Conjecture \ref{conj-crit-pair}, consider the following general setup. Let $G=\Z/p\Z$ with $p\geq 2$ prime and let $A,\, B\subseteq G$ be nonempty subsets with $A+B\neq G$, \
  $|A|\geq |B|$ and \be\label{talkbound}|A+B|=|A|+|B|+r\leq|A|+(1+\alpha)|B|-3,\ee where $\alpha\in (0,1]$ is a real number.
Since $A+B\neq G$, the Cauchy-Davenport Theorem (Theorem \ref{thm-cdt}) ensures $r\geq -1$. As $\alpha\rightarrow 1$, the bound \eqref{talkbound} becomes that from Conjecture \ref{conj-crit-pair}. Note, it is equivalent to $r\leq \alpha |B|-3$ and also to $|B|\geq \alpha^{-1}(r+3)$. Smaller values of $\alpha$ then correspond to a stronger small doubling hypothesis. Sometimes, the hypothesis \eqref{talkbound} is stated without the $-3$, so in the form $|A+B|\leq |A|+(1+\alpha)|B|$.
Particularly in older work, the conclusion regarding $P_C$ is often omitted.  In all such cases,  rectification arguments (see Proposition \ref{prop-3k4-equiv}) can be used to deduce the desired conclusion regarding $P_C$ from the conclusions regarding $P_A$ and $P_B$ combined with Theorem \ref{thm-3k4Z}.
 As such, there is no need to distinguish which results  explicitly state a conclusion regarding $P_C$, and which do not.

Most partial progress towards Conjecture \ref{conj-crit-pair} involves achieving the desired conclusions in Conjecture \ref{conj-crit-pair} by  imposing a stronger small doubling hypothesis $r\leq \alpha|B|-3$ or $r\leq \alpha|B|$ together with an additional density assumption ensuring the sets in question are not too large. There are three major variations for the latter restriction.
The first style has the form $|A|\leq cp$ for some real constant $c>0$.  The second, closely related, has the form $|A+B|\leq cp$ for some real constant $c>0$. Note that $|B|\leq |A|\leq cp$ implies $|A+B|<|A|+(1+\alpha)|B|-3\leq (2+\alpha)cp-3,$ so a restriction of the first type  implies one of the second type.  Finally, in some cases, the density restriction on $|A+B|$ or $|A|$ is shifted to a relative one between  $r$ and $p$, thus having the form $r\leq cp$ for some real constant $c>0$. This removes unnecessary density restrictions from  $A$, $B$ or $A+B$ and instead shifts them to the parameter $r$.
Of course, if the upper bound in \eqref{talkbound} holds with equality, so for $A+B$ having small doubling constant close to the value given by $\alpha$,  then  $|B|\approx\alpha^{-1}(r+3)\leq \alpha^{-1}(cp+3)$ is implied, which is again a density restriction of the first type, though only on the smaller set $B$. However, the value $c$ can be held  constant as $\alpha\rightarrow 0$, so sufficiently small values of $\alpha$ (relative to the fixed constant $c$) lead only to trivial constraints on $|B|$. Indeed, we always have $|B|\leq \frac{p-1}{2}$. Otherwise $|A|\geq |B|\geq \frac{p+1}{2}$ and the Cauchy-Davenport Theorem (Theorem \ref{thm-cdt}) imply $A+B=G$, contrary to hypothesis. This means,  once $\alpha\leq 2c,$ then the bound $r\leq \alpha |B|<\alpha\frac{p}{2}\leq cp$ is guaranteed to hold, so  a restriction of the form $r\leq cp$ includes the case of optimal  density restrictions for  sufficiently small values of $\alpha>0$, as well as including cases with $\alpha$ larger, even up to $1$, having increasingly stronger density restrictions imposed on $B$.

Ideally, one would like to minimize the amount of unnecessary  density restrictions while at the same time allowing  $\alpha$ to be as close to $1$ as possible. These are  competing goals, with improved estimates for one coming at the cost of the estimates for the other. The range of results can generally be grouped as follows.

\textbf{Small $r$.} When $r$ is very small, results more general than Conjecture \ref{conj-crit-pair} are known, often valid in more arbitrary groups $G$ or fully characterizing $A$ and $B$. Along these lines, we always have $r\geq -1$ by the Cauchy-Davenport Theorem (Theorem \ref{thm-cdt}). When $r=-1$, Vosper's Theorem \cite[Theorem 8.1]{vosper} \cite{Grynk-book} implies that Conjecture \ref{conj-crit-pair} holds without restriction. When $r=0$, Hamidoune and Rodseth \cite{ham-rod-vosp+1} established Conjecture \ref{conj-crit-pair} under the additional assumption $7\leq |A|+|B|=|A+B|\leq p-4$. A more general result from \cite{G-kst+1} fully characterizes all $A$ and $B$ with $|A+B|=|A|+|B|$ in all cases.  In particular, Conjecture \ref{conj-crit-pair} is true without restriction for $r=0$. (To derive Conjecture \ref{conj-crit-pair} from \cite[Theorem 4.1]{G-kst+1}, note that the hypotheses of Conjecture \ref{conj-crit-pair} ensure $|A|,|B|,|C|\geq 3$, and we must have a strict inequality in one of these estimates as $p=9$ is not prime. Using Vosper duality (Proposition \ref{lemma-dual}) allows one to w.l.o.g. assume $|C|\geq 4$, and then applying \cite[Theorem 4.1]{G-kst+1} reduces consideration to Vosper's Theorem). This allows us to assume $r\geq 1$ whenever trying to establish a case in Conjecture \ref{conj-crit-pair}. When $r=1$, Hamidoune, Serra and Z\'emor \cite{ham-vosp+2} established Conjecture \ref{conj-crit-pair} additionally assuming $|A|,\,|C|\geq 5$, \ $|B|\geq 4$ and $p\geq 53$.

\textbf{Low Density.} If one is willing to  sacrifice all restraint with regard to the density assumptions, it is possible to achieve the maximal value $\alpha=1$. Along these lines, the rectification results of Bilu, Lev and Ruzsa from \cite[Theorem 3.1]{rectification}  combined with Theorem \ref{thm-3k4Z} ensure Conjecture \ref{conj-crit-pair} is true additionally assuming $|A\cup B|\leq \log_4 p$. These were later improved as a consequence of a result of  Lev \cite{lev-rect-threshold} (combined with \cite[Theorem 20.2]{Grynk-book}) to $|A\cup B|\leq \lceil\log_2 p\rceil$ (see Theorem \ref{thm-smallr}). When $A=B$, taking into account that $A+A$ has small sumset, alternative rectification arguments of Green and Ruzsa \cite{Green-Ruza-rect} combined with Theorem \ref{thm-3k4Z} ensure Conjecture \ref{conj-crit-pair} is true  assuming $A=B$ with $|A|< cp$, where $c=(1/96)^{108}$. For general sumsets $A+B$, the added assumption $r\leq cp-1.2$ suffices, where $c=1.4\cdot 10^{-63\,951}$ \cite[Theoreom 21.8]{Grynk-book} (note there is an error in the explicit statement of the effectively existential constant $c$ in \cite[Theoreom 21.8]{Grynk-book} owing to a typographic error when stating  the rectification threshold of Green and Ruzsa).  Finally,  for smaller $r$,  there is a similar result \cite[Theorem 21.7]{Grynk-book} showing the logarithmic bound $r+2\leq \frac{5}{16}\lceil\log_2 p\rceil$ suffices.

\textbf{Mid-Range Density.} In this range,  a balanced approach is taken aiming for tangible (rather than effectively existential) estimates for both the small doubling  and density constraints simultaneously. The original result of Freiman \cite{Freiman-vosp} \cite{Freiman-monograph} \cite{natbook} is the prototype here.  It shows Conjecture \ref{conj-crit-pair} true for $A=B$ assuming a flexible pair of  assumptions $r\leq \alpha |A|-3$ and $|A|\leq cp$, with the values of $c$ and $\alpha$ dependent on each other. Despite the variation allowed, the choice of assumptions $|2A|\leq 2.4|A|-3$ and $|A|\leq \frac{p}{35}$ fixed in   \cite[Theorem 2.11]{natbook} is now  emblematic with the result itself. Rodseth \cite{rodseth-2.4} improved Freiman's estimates to  $|2A|\leq 2.4|A|-3$ and $|A|\leq \frac{p}{10.7}$. A range of similar results based on these methods are presented  in \cite[Chapter 19]{Grynk-book} with the density assumptions of the form $|2A|\leq cp$ with $c\leq \frac12$. Candela, Serra and Spiegel showed
 $|2A|\leq 2.48|A|-7$ and $|A|<p/10^{10}$ \cite{pablo-serra-3k4} suffices for $A=B$. Lev and Shkredov \cite{lev-3k-4-highenrg} showed
 $|2A|< 2.59|A|-3$ and $101\leq |A|<0.0045p$ suffices for $A=B$, and that  $|A-A|<2.6|A|-3$ and $|A|<0.0045p$ suffices for $A=-B$. Lev and Serra \cite{lev-serra-2.7} showed
 $|2A|<2.7652|A|-3$ and $10\leq |A|< 1.25\times 10^{-6}p$ suffices for $A=B$. All these results are for $A=B$ or $A=-B$. For general sumsets $A+B$, there are results of this type given in \cite[Chapter 19]{Grynk-book} provided $|A|$ and $|B|$ are very close in size. For example, $|B|\leq |A|\leq \frac{4}{3}$, \ $|A+B|\leq |A|+1.05|B|-3$ and $|A+B|\leq 0.0044p$. Huicochea achieved similar results with more relaxed constraints on the relative sizes of $|A|$ and $|B|$ , e.g., $10^3|A|^{\frac23}\leq |B|\leq |A|$, \ $|A+B|<|A|+1.03|B|$ and $|A|<0.0045p$ \cite{huicochea-3k-4-distinct}, with improved bounds when $|A|/|B|$  is large.

\textbf{High Density.} Many results in the mid-range density range have a bias towards focussing on the bound $r\leq \alpha |B|-3$ while simultaneously having a ``reasonable'' density restriction. However, applications often  require working with high density sets. An alternative focus is to achieve as a high a density as possible while simultaneously having ``reasonable'' small doubling constraints.   The work of Candela, Gonz\'alez-S\'anchez and  Grynkiewicz  \cite{pablo-grynk-3k-4-A=B} is the example in this range. There, Conjecture \ref{conj-crit-pair} is established for $A=B$  when $|2A|\leq 2.136|A|-3$ and $|2A|\leq \frac34p$, with improved estimates for the small doubling assumption when $|A|$ is close to $\frac13p$.

\textbf{Optimal Density.} The goal here is to sacrifice restraint with $\alpha$ in order to achieve density constraints of the optimal conjectured order of magnitude   $|A+B|\leq p-O(r)$.
Per earlier discussion, results with relative constraints between $r$ and $p$ yield results with no unnecessary density restrictions (for very small values of $\alpha$). Thus the two results mentioned in the low density section, of the form $r\leq cp-1.2$ and $r+2\leq \frac{5}{16}\lceil \log_2 p\rceil$,  also yield  optimal density results for general sumsets $A+B$ for \emph{very} small $\alpha$. However, for $A=B$, the  gold standard goes to Serra and Z\'emor \cite{serra-3k-4-plunnecke} who showed Conjecture \ref{conj-crit-pair} to be true needing only the additional assumptions  $|2A|\leq 2.0001|A|$, \ $p>2^{94}$ and $|A|\geq 4$ (this condition ensures that either $|2A|\leq 3|A|-4$ or that the case $r\leq 0$ is applicable).

\medskip

The focus of this paper regarding  Conjecture \ref{conj-crit-pair} is on general sumsets $A+B$ in the high and optimal density ranges, where there  are currently \emph{no}  results having  tangible estimates for the small doubling constraint. Indeed, without constraints on the relative sizes of $|A|$ and $|B|$, there are no results of any sort having tangible estimates for the small doubling constraint for general sumsets $A+B$.
The high density result from \cite{pablo-grynk-3k-4-A=B} was achieved by means of the original exponential sum argument of Freiman (using Rodseth's improved handling of the exponential sum calculations) combined with a new usage of Vosper Duality to reduce the density restriction imposed by the original combinatorial reduction portion of the argument.
We will show (in Proposition \ref{lemma-four-case1}) how this Vosper Duality argument can be adapted to general sumsets $A+B$ and then combine it with a new adaptation of  exponential sum arguments to general sumsets $A+B$. This will give us a result (Theorem \ref{fouier-II-distinct}) with tangible estimates for the small doubling constant assuming the sizes of $|A|$ and $|B|$ are very close. Combining this with  isoperimetric reduction methods will then allow us to remove most density restrictions as well as the restrictions on the sizes  of $|A|$ and $|B|$, yielding the following result. Note Theorem \ref{thm-p-O(r)} will be derived as a simple consequence of  Theorem \ref{thm-r-O(p)-a}, whose statement is slightly more technical but more general.

\begin{theorem}\label{thm-p-O(r)}
Let $G=\Z/p\Z$ with $p\geq 2$ a prime, let $A,\, B\subseteq G$ be nonempty subsets with $A+B\neq G$,  and set $C=-\,G\setminus (A+B)$. Suppose
  $|A|\geq |B|$   and  \begin{align}\nn&|A+B|=|A|+|B|+r\leq \min\{|A|+1.0527|B|-3,\quad p-9(r+3)\}.\end{align}
Then there are arithmetic progressions $P_A,\,P_B,\,P_{C}\subseteq G$ all having  common difference  such that
\begin{align*} &X\subseteq P_X\;\und\; |P_X|\leq |X|+r+1\quad \mbox{ for all $X\in \{A,B,C\}$}.\end{align*}
\end{theorem}

Any density assumption of the form $|A+B|\leq p-c(r+3)$, for some $c\geq 1$, can be viewed as a version of Conjecture \ref{conj-crit-pair} valid with an additional  $(1-\epsilon)$-density restriction on $A+B$. Indeed, for any $\epsilon>0$, there then exists some $\alpha>0$, namely $\alpha=\frac{2\epsilon}{c}$, such that Conjecture \ref{conj-crit-pair} holds under the additional  assumptions $$|A+B|\leq |A|+(1+\alpha)|B|-3=|A|+\big(1+\frac{2\epsilon}{c}\big)|B|-3\;\und\;|A+B|\leq (1-\epsilon)p.$$ This can be seen since $A+B\neq G$ and $|A|\geq |B|$ ensures $|B|\leq \frac{p}{2}$, while $|A+B|=|A|+|B|+r\leq |A|+(1+\alpha)|B|-3$ is equivalent to $r+3\leq \alpha|B|$. Thus  $|A+B|\leq (1-\epsilon)p=p-\frac{c\alpha}{2}p\leq p-c\alpha|B|\leq p-c(r+3)$, allowing the result of the form $|A+B|\leq p-c(r+3)$ to be applied.
As such, Theorem \ref{thm-p-O(r)} allows densities arbitrarily close to $1$ at the cost of an increasingly stronger small doubling hypothesis $\alpha\leq \frac29\epsilon$ (for $\epsilon\leq .23715$).

For those that would prefer the $-3$ be removed from the small sumset hypothesis in Theorem \ref{thm-p-O(r)}, this can be derived from Theorem \ref{thm-p-O(r)}, for sufficiently large $p$,  by imposing a slightly more restrictive small doubling hypothesis. For instance, if one assumes $|A|\geq |B|\geq 3+\delta$, where $\delta=1$ if $B$ is a translate of  $A$  and otherwise $\delta=0$,  and that \be\label{titat}|A+B|=|A|+|B|+r\leq \min\{|A|+1.05|B|,\quad p-9(r+3)\},\ee then Theorem \ref{thm-p-O(r)} can be applied except when $3>(.0027)|B|\geq (.0027)(.05)^{-1}r$, implying $r\leq 55$.
However, as mentioned in the description of results for low density, Conjecture \ref{conj-crit-pair} is known when $r+2\leq \frac{5}{16}\lceil \log_2 p\rceil$. Thus for $p> 2^{182}$, the cases $r\leq 55$ are known, meaning Conjecture \ref{conj-crit-pair} holds when additionally assuming $|B|\geq 3+\delta$, \eqref{titat} and $p>2^{182}$ (the hypothesis $|B|\geq 3+\delta$ is needed to ensure \eqref{titat} implies either $|A+B|\leq |A|+2|B|-4$ or that the mentioned  $r\leq 0$ case in Conjecture \ref{conj-crit-pair} is available) .

If one wishes to remove the $-3$ from the small sumset hypothesis in Theorem \ref{thm-p-O(r)} and have a result valid without an additional assumption that $p$ be large, this can also be derived from Theorem \ref{thm-p-O(r)} by imposing a more restrictive small doubling hypothesis.
For instance,
if one assumes $|A|\geq |B|\geq 3+\delta$ as before and that \be\label{titat-b}|A+B|=|A|+|B|+r\leq \min\{|A|+1.021|B|,\quad p-9(r+3)\},\ee then Theorem \ref{thm-p-O(r)} can be applied except when $3>(.0317)|B|\geq (.0317)(.021)^{-1}r$, implying $r\leq 1$. If $|B|\leq 4$, then \eqref{titat-b} implies $r\leq 0$, in which case  Conjecture \ref{conj-crit-pair} is known. As mentioned in the description of results for small $r$, Conjecture \ref{conj-crit-pair} is also known when $r=1$ provided $|B|\geq 5$, $|A+B|\leq p-5$ and $p\geq 53$. However, for $r=1$, we have $|B|\geq (.021)^{-1}r>47$ with $2\cdot 47+1\leq |A|+|B|+r=|A+B|\leq p-9(r+3)=p-36$, ensuring $p\geq 131$, so that all extra conditions needed for the case $r=1$ hold. This means Conjecture \ref{conj-crit-pair} holds when additionally assuming $|B|\geq 3+\delta$ and \eqref{titat-b}.

\medskip

Theorem \ref{thm-p-O(r)}   has a similar  order of magnitude $|A+B|\leq p-O(r)$ density restriction as that for the result of Serra and Z\'emor \cite{serra-3k-4-plunnecke}. However, Theorem \ref{thm-p-O(r)} is valid with a significantly better small doubling constant $0.0527$ as opposed to $0.0001$, allows for general sumsets $A+B$, and does not require an added restriction that $p$ be sufficiently large. Unfortunately, it does not achieve  optimal density like the result of Serra and Z\'emor. Using Proposition \ref{lemma-four-case1}, the methods of Lev and Shkredov \cite{lev-3k-4-highenrg} easily adapt to yield a flexible high density version of their result, stated in Theorem \ref{thm-lev-sh-fixed}.
 In the final portion of the paper, we  utilize Theorems \ref{thm-r-O(p)-a} and Theorem \ref{thm-lev-sh-fixed}, as well as several other recent advances \cite{huicochea-reduction-3k-4} \cite{lev-3k-4-highenrg} \cite{Grynk-book} \cite{hyperatoms}, in order to extend and improve the methods of Serra and Z\'emor \cite{serra-3k-4-plunnecke}, resulting in Theorem \ref{thm-ideal-density}. Like the result of Serra and Z\'emor \cite{serra-3k-4-plunnecke}, this gives a version of the $3k-4$ Theorem in $\Z/p\Z$  valid in the ideal density range. However, the result here is valid for general sumsets $A+B$ rather than only $A+A$, does not require $p$ be sufficiently large, and has a noticeably  improved small doubling constant $0.01$ versus $0.0001$.  As with Theorem \ref{thm-p-O(r)}, the $-3$ can also be removed by imposing a more stringent small doubling hypothesis, replacing the bound $|A+B|\leq |A|+1.01|B|-3$ by the bounds
$$|A+B|\leq |A|+1.004|B|\;\und\; |B|\geq 3+\delta,$$ where $\delta=1$ if $B$ is a translate of either $A$ or $-\,G\setminus (A+B)$ and otherwise $\delta=0$, or alternatively by replacing the bound $|A+B|\leq |A|+1.01|B|-3$ by the bounds $$|A+B|\leq |A|+1.009|B|,\quad |B|\geq 3+\delta\quad\;\und\;p>2^{89}.$$

\begin{theorem}\label{thm-ideal-density}
Let $G=\Z/p\Z$ with $p\geq 2$ prime, let $A,\, B\subseteq G$ be nonempty subsets with $A+B\neq G$, and set $C=-\,G\setminus (A+B)$. Suppose
  $|A|\geq |B|$   and \be\nn |A+B|=|A|+|B|+r\leq \min\{|A|+1.01|B|-3,\quad p-r-3\}.\ee
Then there are arithmetic progressions $P_A,\,P_B,\,P_{C}\subseteq G$ all having  common difference  such that
\begin{align*} &X\subseteq P_X\;\und\; |P_X|\leq |X|+r+1\quad \mbox{ for all $X\in \{A,B,C\}$}.\end{align*}
\end{theorem}

\section{Preliminaries}

We overview the basic background and results needed for the proofs. Let $G$ be an abelian group and let $A,B\subseteq G$. For a subset $I\subseteq \Z$ and $d\in G$, we let
$$I\cdot d=\{xd:\; x\in I\}\subseteq G.$$
We let $\mathsf H(A)=\{x\in G:\; x+A=A\}\leq G$ denote the stabilizer of $A$, which is a subgroup of $G$.

For $x\in \Z$, let $\exp(x)=e^{2\pi i x}\in \mathbb C,$ where $i$ is the square root of $-1$, which is a complex point contained on the unit circle whose value depends only on $x$ modulo $1$.
An open half-arc of the complex unit circle is a set of the form $C=\{\exp(x): \;x\in (u,u+\frac12)\}$ for some $u\in \R$.  Given a complex number $z\in \mathbb C$, we let $\overline z$ denote its complex conjugate.
For $p\geq 2$ prime, the value of $\exp(a/p)$  depends only on the value of $a$ modulo $p$, making $\exp(a/p)$ will defined for $a\in \Z/p\Z$.
Given a subset $A\subseteq \Z/p\Z$ and $x\in \Z$, we let $S_A(x)=\Summ{a\in A}\exp(ax/p)$. See \cite[Chapter 19]{Grynk-book} for basic results connecting the exponential sums $S_A(x)$ with sumsets. In particular, the following basic lemma of Freiman (see \cite[Lemma 19.1]{Grynk-book}) provides a crucial link between exponential sums and Additive Combinatorics.


\begin{lemma}\label{lem-freiman-circle}
If $z_1,\ldots, z_N\in \mathbb C$ is a sequence of  points lying on the complex unit circle  such that every
open half-arc contains at most $n$ of the terms $z_i$ for $i\in [1,N]$, then $|\Sum{i=1}{N}z_i|\leq 2n-N$.
\end{lemma}

\subsection*{Basic Sumset Results}

When $G=\Z/p\Z$ with $p\geq 2$ prime, the basic lower bound for sumsets is the Cauchy-Davenport Theorem. See \cite[Theorem 6.2]{Grynk-book}.

\begin{theorem}[Cauchy-Davenport Theorem]
\label{thm-cdt} Let $G=\Z/p\Z$ with $p\geq 2$ prime and let $A,B\subseteq G$ be nonempty. Then $|A+B|\geq \min \{p,\,|A|+|B|-1\}$.
\end{theorem}

The generalization to more general abelian groups is Kneser's Theorem \cite[Theorem 6.1]{Grynk-book} \cite[Theorem 4.1]{natbook} \cite[Theorem 5.5]{taovu}.

\begin{theorem}[Kneser's Theorem]
\label{thm-kt} Let $G$ be an abelian group and let $A,B\subseteq G$ be finite and nonempty. Then $|A+B|\geq |H+A|+|H+B|-|H|$, where $H=\mathsf H(A+B)$.
\end{theorem}

If $G$ is a finite abelian group and $A\subseteq G$, then  \be\label{AAcomp}(A-A^{\mathsf c})^\mathsf c=\mathsf H(A).\ee Indeed, it is immediate that $0\notin A-A^\mathsf c$, and then, since $|A|+|A^\mathsf c|=|G|$, Kneser's Theorem implies that $A-A^\mathsf c=G\setminus H$, where $H=\mathsf H(A-A^\mathsf c)$. To see $H=\mathsf H(A)$, note that if $x\in A$ with $x+H\nsubseteq A$, then $x+H$ intersects both $A$ and $A^\mathsf c$ with $|(x+H)\setminus A|+|(x+H)\setminus A^\mathsf c|=|H|$, whence $|A-A^{\mathsf c}|\geq |H+A|+|H+A^\mathsf c|-|H|\geq (|A|+|A^\mathsf c|+|H|)-|H|=|G|$, contradicting that $0\notin A-A^\mathsf c$. Thus $\mathsf H(A)\leq H=\mathsf H(A-A^\mathsf c)$ with the reverse inclusion trivial.




The following are Petridis' version of the Ruzsa-Pl\"unnecke Bounds \cite[Theorem 1.6.1]{ruzsa-alfred-book}. Theorem \ref{thm-plunnecke-petridis}.1 is \cite[Theorem 1.5]{petridis-plunnecke} and Theorem \ref{thm-plunnecke-petridis}.2 is \cite[Theorem 3.1]{petridis-plunnecke}.
In all cases, the explicit definition of $A'$ (as given below) is not found in the statement of the theorems from \cite{petridis-plunnecke} but is how $A'$ is defined in the proof of these theorems.

\begin{theorem}[Petridis' Ruzsa-Pl\"unnecke Bounds] \label{thm-plunnecke-petridis}
Let $G$ be an abelian group, let  $A,\,B\subseteq G$ be finite, nonempty  subsets and let $A'\subseteq A$ be a nonempty subset attaining the minimum $$\alpha:=\min\Big\{\frac{|A'+B|}{|A'|}:\; \emptyset\neq A'\subseteq A\Big\}\leq \frac{|A+B|}{|A|}.$$ Then the following all hold.
\begin{itemize}
\item[1.] $|C+A'+B|\leq \alpha|C+A'|$ for all finite subsets $C\subseteq G$.
\item[2.] $|A'+nB|\leq \alpha^{n}|A'|\leq\alpha^n|A|$ for all $n\geq 0$.
\end{itemize}
\end{theorem}



\subsection*{Additive Isomorphism}

For $d\in G$, we let $$\mbox{$\ell_d(A)$ be the minimal length of an arithmetic progression with difference $d$ containing $A$}.$$ Note $\ell_d(A)=\infty$ if no such progression exists.
%
Let $G$ and $G'$ be abelian groups and let $A,B\subseteq G$ be finite and nonempty.
A \textbf{Freiman homomorphism} is a map $\varphi: A+B\rightarrow G'$ for which there are coordinate maps $\varphi_A:A\rightarrow G'$ and $\varphi_B:B\rightarrow G'$ such that $\varphi(a+b)=\varphi_A(a)+\varphi_B(b)$ for all $a\in A$ and $b\in B$. An injective Freiman homomorphism $\varphi:A+B\rightarrow G'$ defines a \textbf{Freiman isomorphism} $\varphi:A+B \rightarrow A'+B'\subseteq G'$, where $A'=\varphi_A(A)$ and $B'=\varphi_B(B)$, in which case we say the sumsets $A+B\subseteq G$ and $A'+B'\subseteq G'$ are \textbf{isomorphic}, denoted $A+B\cong A'+B'$. There is often no loss of generality to replace $A$ and $B$ with appropriate translates so that $0\in A\cap B$. Likewise for $A'$ and $B'$.  If $0\in A\cap B$ and $\varphi_A(0)=\varphi_B(0)=0$, we call $\varphi$ a \textbf{normalized} Freiman homomorphism. If $A+B\cong A'+B'$ and $0\in A\cap B$, then replacing $A'$ and $B'$ by appropriate translates  with  $0\in A'\cap B'$ (namely $-\psi_A(0)+A'$ and $-\psi_B(0)+B'$), there is a normalized Freiman isomorphism $\psi:A+B\rightarrow A'+B'$. Since any normalized Freiman homomorphism must have its coordinate maps agree on any element common to their domains, this ensures that $A,\,B\subseteq G$ are translates of each other in $G$ if and only if $A',B'\subseteq G'$ are translate of each other in $G'$, for isomorphic sumsets $A+B\cong A'+B'$.
Any group homomorphism $\varphi:G\rightarrow G'$ restricts to a Freiman homomorphism of $A+B$ which is an isomorphism when the restriction $\varphi\mid_{A+B}$ is injective, and is normalized when $0\in A\cap B$. Every sumset $A'+B'\subseteq G'$ with $0\in A'\cap B'$ has a universal ambient group $G$ containing an isomorphic sumset $A+B$ for which any normalized Freiman homomorphism $\varphi:A+B\rightarrow A'+B'$ can be extended to a group homomorphism $\overline \varphi:G\rightarrow G'$. See \cite[Chapter 20]{Grynk-book} \cite[Chapter 5.5]{taovu} for a fuller discussion of the  basic theory of Freiman homomorphisms, including the details mentioned above (the latter reference only treats $A=B$).

An example of when a well-behaved isomorphism exists is as follows.
We say that the sumset $A'+B'\subseteq G'$ \textbf{rectifies} if
 there is some $d\in G'$ such that \be\label{rect-threshold}\ell_d(A')+\ell_d(B')\leq \ord(d)+1.\ee  In such case,  $A'\subseteq a_0+[0,M-1]\cdot d$ and $B'\subseteq b_0+[0,N-1]\cdot d$ for some $a_0,b_0\in G'$ and  $M,N\geq 1$ with $M+N\leq \ord(d)+1$, and then the maps $\varphi_A:\Z\rightarrow G'$ and $\varphi_B:\Z\rightarrow G'$ given by $\varphi_A(x)=a_0+xd$ and $\varphi_B(y)= b_0+yd$ induce a Freiman isomorphism $\varphi: A+B\subseteq \Z\rightarrow A'+B'$, where $A=\varphi_A^{-1}(A')\cap [0,n-1]\subseteq \Z$ and $B=\varphi_B^{-1}(B')\cap [0,n-1]\subseteq \Z$ with $n=\ord(d)$.

\subsection*{Vosper Duality and Additive Trios} We continue with the rather crucial notion of Vosper Duality (in basic form, this can be traced to \cite{vosper}), which we will afterwards formulate in the more abstract terms of  additive trios (introduced in \cite{trio-origin}).
 If $(\{x\}\cup A)+B\neq A+B$ for all $x\in A^\mathsf c$, then we say that $A$ is \textbf{saturated} in the sumset $A+B$. The following basic observation will be crucial. Its proof is straightforward. See  \cite[Lemma 7.2]{Grynk-book}.

\begin{proposition}[Vosper Duality]\label{lemma-dual}
Let $G$ be an abelian group and let $A,\,B\subseteq G$ be subsets. Then $$-(A+B)^\mathsf c+B\subseteq -A^\mathsf c$$ with equality holding if and only if $A$ is saturated in $A+B$. When this is the case with $|A+B|=|A|+|B|+r$ and $G$ finite, then $|-(A+B)^\mathsf c+B|=|(A+B)^\mathsf c|+|B|+r$ with $-(A+B)^\mathsf c$ saturated in $-(A+B)^\mathsf c+B$.
\end{proposition}

An \textbf{additive trio} in a finite abelian group $G$  is a triple $\mathcal T=(A,B,C)$ with $A,\,B,\,C\subseteq G$ \emph{nonempty} subsets satisfying $A+B+C\neq G$. Thus $z\notin A+B+C$ for some $z\in G$. In such case,  \be\label{inclusion}A\subseteq z-(B+C)^\mathsf c,\; B\subseteq z-(A+C)^\mathsf c\;\und\; C\subseteq z-(A+B)^\mathsf c,\quad\mbox{ for any $z\in (A+B+C)^\mathsf c$}.\ee  We say that $A$ is \textbf{saturated} in the trio $\mathcal T$ if $(A\cup \{x\})+B+C=G$ for every $x\in A^\mathsf c$, with $B$ and $C$ being saturated defined analogously. The trio $\mathcal T$ is \textbf{saturated} if $A$, $B$ and $C$ are each saturated in $\mathcal T$. The set $z-(B+C)^\mathsf c$ is precisely the set of those elements $x\in G$ for which $z\notin (A\cup \{x\})+B+C$. From this observation, it readily follows that $A$ is saturated in $\mathcal T$ if and only if $A=z-(B+C)^\mathsf c$ for all $z\in (A+B+C)^\mathsf c$. Likewise, equality holding in the other inclusions in \eqref{inclusion} are equivalent to $B$ or $C$ being saturated in $\mathcal T$.
As a consequence, if $A$ is saturated in $\mathcal T$, then \eqref{AAcomp} implies $A+B+C=G\setminus (z+H)$ with $H=\mathsf H(A)=\mathsf H(B+C)$. 
For an additive trio, we define $$r=\mathsf r\,(\mathcal T)=|G|-|A|-|B|-|C|.$$
When $G=\Z/p\Z$, the condition $A+B+C\neq G$ together with  the Cauchy-Davenport Theorem (Theorem \ref{thm-cdt}) ensures $|B+C|\geq |B|+|C|-1$, in which case \eqref{inclusion} implies $|A|\leq |G|-|B+C|\leq |G|-|B|-|C|+1$, meaning $r=\mathsf r\,(\mathcal T)\geq -1$. Note $(A,B,-(A+B)^\mathsf c)$ is always an additive trio with $$\mathsf r\big(A,B,-(A+B)^\mathsf c\big)=r\quad\mbox{ where $|A+B|=|A|+|B|+r$},$$so long as $A$ and $B$ are nonempty with $A+B\neq G$, as exhibited by the element $z=0$.

If $\mathcal T=(A,B,C)$ is not saturated, we can successively saturate the sets by replacing $C$ by $C'=z-(A+B)^\mathsf c$, for any $z\in (A+B+C)^\mathsf c$, then replace $B$ with the maximal set $B'$ such that $B\subseteq B'$ and  $A+B'=A+B$. Note $B'$ exists since $A+B$ is finite.   In view of $C'=z-(A+B)^\mathsf c$, $A+B$ is the maximal subset such that $A+B+C'=(A+B)+(z-(A+B)^\mathsf c)=G\setminus (z+H)$, where $H=\mathsf H(A+B)$, which makes $B'\subseteq G$ the maximal subset with $A+B'+C'\neq G$, while $C'=z-(A+B)^\mathsf c=z-(A+B')^\mathsf c$.  It follows that $B'$ and $C'$ are saturated in the trio $(A,B',C')$. Finally, replacing $A$ with the maximal subset $A'$ such that
$A\subseteq A'$ and $A'+B'=A+B'=A+B$, we likewise obtain a saturated trio $\mathcal T'=(A',B',C')$ with $A\subseteq A'$, $B\subseteq B'$ and $C\subseteq C'$ (the latter in view of \eqref{inclusion}), which we refer to as a successive saturation of $\mathcal T$. Such a process can be done using any $z\in (A+B+C)^\mathsf c$ and any ordering of the sets $A$, $B$ and $C$.

The hypotheses of Conjecture \ref{conj-crit-pair} ensure that $\mathcal T=(A,B,C)$ is an additive trio in $G=\Z/p\Z$ with
\be\label{trio-form}r=\mathsf r(\mathcal T)\geq -1\;\und\;|X|\geq r+3+\delta_X\quad\mbox{ for every $X\in \{A,B,C\}$},\ee where $A$, $B$, $C$, $r$ and $\delta_X$ are all as defined in Conjecture \ref{conj-crit-pair}.
Conversely, given any additive trio $\mathcal T=(A,B,C)$ with $\mathsf r(A,B,C)\leq r$, we can pass to a saturated trio $(A',B',C')$ as described above and translate to w.l.o.g. assume $0\notin A'+B'+C'$. Then $A'+B'=-(C')^\mathsf c$, \ $A'+C'=-(B')^\mathsf c$, $B'+C'=-(A')^\mathsf c$ and $\mathsf r(A',B',C')\leq r-\Delta$, where $\Delta=(|A'|-|A|)+(|B'|-|B|)+(|C'|-|C|)$.  Moreover any arithmetic progression $P_{X}$ containing $X'\in \{A',B',C'\}$ with $|P_X|\leq |X'|+r-\Delta+1$ also contains the corresponding $X\in \{A,B,C\}$ with $|P_X|\leq |X|+r+1$.
Thus the hypotheses in \eqref{trio-form} are the more symmetric  equivalent formulation for the hypotheses in Conjecture \ref{conj-crit-pair}.

\subsection*{Isoperimetric Machinery}

The Ruzsa-Pl\"unnecke bounds involve multiplicative expansion, considering a subset $X\subseteq A$ that minimizes $\frac{|X+B|}{|X|}$. The isoperimetric method instead considers the additive expansion $|X+B|-|X|$. The basic setup is as follows. Let $G$ be an abelian group and let $B\subseteq G$ be a finite, nonempty subset. For an integer $k\geq 1$, we say that $B$ is \textbf{$k$-separable} if there is a finite subset $X\subseteq G$ with $|X|\geq k$ and $|(X+B)^\mathsf c|\geq k$. When this is the case, we define
\be\label{kappa-k}\kappa_k(B)=\min\{|X+B|-|X|:\; X\subseteq G,\; |X|\geq k,\; |(X+B)^\mathsf c|\geq k\}.\ee
A set $X\subseteq G$ with $|X|, \, |(X+B)^\mathsf c|\geq k$ attaining the minimum in \eqref{kappa-k} is called a \textbf{$k$-fragment} of $B$, and a minimal cardinality $k$-fragment is called a \textbf{$k$-atom}. We let $\alpha_k(B)$ denote the cardinality of a $k$-atom of $B$. By its definition, we have \be\label{iso-bound}|Y+B|\geq \min\{|G|-k+1,\;|Y|+\kappa_k(B)\}\ee for any subset $Y\subseteq G$ with $|Y|\geq k$.
If $X$ is a $k$-atom for $B$, then (by \cite[Lemma 21.1]{Grynk-book}) \be\label{HX-HXB}\mathsf H(X)=\mathsf H(X+B).\ee Letting $H=\mathsf H(X)$ and $k'=\lceil k/|H|\rceil$, it then follows (by \cite[Proposition 21.5]{Grynk-book}) that \be\label{modatoms}\mbox{$\overline B$ is $k'$-separable with $\overline X$ a $k'$-atom for $\overline B$},\ee where $\overline x$ denotes the image of $x\in G$ modulo $H$.
The fundamental property of atoms (originally due to Hamidoune) is the following (see \cite[Theorem 21.1]{Grynk-book}).

\begin{theorem}[Fundamental Theorem for Atoms]\label{atom-fund}Let $G$ be an abelian group, let $k\geq 1$, let $B\subseteq G$ be a finite, nonempty subset, and suppose $B$ is $k$-separable. Let $X$ be a $k$-atom of $B$ and let $F$ be a $k$-fragment of $B$. If $|X\cap F|\geq k$, then $X\subseteq F$. In particular, two distinct $k$-atoms intersect in  at most $k-1$ elements.
\end{theorem}

The following is  a useful upper bound for the size of a $2$-atom \cite[Theorem 8]{hyperatoms}. Note, if $\la B-B\ra=H$ is a proper, finite subgroup of $G$, then the $2$-atoms of $B$ are precisely the $H$-cosets in $G$, so the hypotheses of Theorem \ref{thm-2atombound} ensure $\la B-B\ra$ is either infinite or equal to $G$. This ensures that the case $\la B-B\ra=G$ implies the general case stated in Theorem \ref{thm-2atombound}.

\begin{theorem}\label{thm-2atombound}
Let $G$ be an abelian group and let $B\subseteq G$ be a finite $2$-separable set with $r:=\kappa_2(B)-|B|$. Suppose $B$ has a $2$-atom that is not a coset. Then $\alpha_2(B)\leq r+3$.
\end{theorem}

When $G=\Z/p\Z$ with $|B|$ not extremely large, we can improve the bound  in Theorem \ref{thm-2atombound}.

\begin{theorem}\label{iso3}
 Let $G=\Z/p\Z$ with $p\geq 2$ prime, and let $B\subseteq G$  with $|B|\geq 2$. Suppose $B$ is $2$-separable and let  $r\geq \kappa_2(B)-|B|$ be an integer.
\begin{itemize}
\item[1.] If $|B|\leq p-3r-4,$ then  $\alpha_2(B)\leq 2+\sqrt{2r+2}$.
\item[2.] If $|B|\leq p-3r-5$ and $r\geq 1$, then $\alpha_2(B)\leq \frac12+\sqrt{2r+\frac{17}{4}}$.
\end{itemize}
\end{theorem}

\begin{proof}
Part 1 is \cite[Theorem 21.6]{Grynk-book}. Note $r\geq r':=\kappa_2(B)-|B|\geq -1$ follows from the Cauchy-Davenport Theorem (Theorem \ref{thm-cdt}).  Since $r\geq 1$, the bound given in Part 2 is always at least $3$. As a result, if $r'\leq 0$, then
Part 1 (applied with $r=r'$) yields the desired bound for Part 2. This allows us to assume $r'=r$. Then
the proof of Part 1 given in  \cite[Theorem 21.5]{Grynk-book} also derives Part 2. Note---in the notation of \cite[Theorem 21.5]{Grynk-book}, where their  $r$ equals our $r+1$---that the assumptions $|B|\leq p-3r-1$, \ $|X|\geq 5$ and $r\geq 3$ derived at the beginning of the proof get replaced by $|B|\leq p-3r-2$, \ $|X|\geq 4$ and $r\geq 2$, leading with bare minimal modification in the calculations to the conclusion $|X|\leq \frac12+\sqrt{2r+\frac94}$ stated at the end of the proof, which replacing $r$ by $r+1$ gives the bound stated here.
\end{proof}

Per the arguments in \cite[Lemma 2.3]{serra-3k-4-plunnecke}, we can obtain the following upper bound for the size of a $k$-atom in terms of the size of a $2$-atom. The result in \cite[Lemma 2.3]{serra-3k-4-plunnecke} was only given for $G=\Z/p\Z$ with $p$ prime.

\begin{theorem}\label{thm-katom-bound}
Let $G$ be a finite abelian group, let $k\geq 2$, and let $B\subseteq G$ be a $k$-separable subset with $r=\kappa_k(B)-|B|$ and $\la B-B\ra=G$. Suppose $B$ has a $k$-atom $X$ that is not a coset. Then $$\alpha_k(B)=|X|\leq \lceil k/|H|\rceil \cdot |H|-|H|+r+|H|\alpha_2(\overline X)\leq k-1+r+|H|\alpha_2(\overline X),$$ where $H=\mathsf H(X)$ and $\overline X\subseteq G/H$ denotes the image of $X$ modulo $H$.
\end{theorem}

\begin{proof}
Let $X$ be a $k$-atom of $B$, let $H=\mathsf H(X)$ and let $\overline x\in G/H$ for $x\in G$ denote the image of $x$ modulo $H$. By hypothesis,  $|X+B|=|X|+|B|+r\leq |G|-k$. Note $|X|\geq k\geq 2$ ensures that $G$ is nontrivial, which combined with the hypothesis $\la B-B\ra=G$ forces $|B|\geq 2$. Since $k\geq 2$ and $|B|\geq 2$, it follows that $X$ is $2$-separable with $$\kappa_2(X)\leq |X+B|-|B|=|X|+r.$$ Let $Y$ be a $2$-atom of $X$ translated so that $0\in Y$. Since $|Y|\geq 2$, let $d\in Y$ be any nonzero element. Thus \be\label{peasoup}|X\cup (d+X)|=|\{0,d\}+X|\leq |Y+X|=\kappa_2(X)+|Y|\leq |X|+|Y|+r.\ee Since $X$ and $d+X$ are both $k$-atoms of $B$, it follows from Theorem \ref{atom-fund} that either $d+X=X$ or $|(d+X)\cap X|\leq k-1$.

If $H=\mathsf H(X)$ is trivial,  the latter case must hold (as $d+X=X$ with $d\neq 0$ ensures that $d\in H$ with $H$ nontrivial).  Thus  $|X\cup (d+X)|=2|X|-|X\cap (d+X)|\geq 2|X|-k+1$, which combined with \eqref{peasoup} yields the desired bound $\alpha_k(B)=|X|\leq k-1+r+|Y|=k-1+r+\alpha_2(X)=k-1+r+|H|\alpha_2(\overline X)$. Therefore we may instead assume $H$ is nontrivial.

 By \eqref{HX-HXB}, we have $\mathsf H(X+B)=\mathsf H(X)=H$. Let $k'=\lceil k/|H|\rceil$. By Kneser's Theorem (Theorem \ref{thm-kt}), $|G|-k-|H|\geq |X+B|-|X|\geq |H+B|-|H|$. As result, if $|H|\geq k$, then $H$ is a $k$-fragment of $B$ with $|X+B|-|X|= |H+B|-|H|$, and then, since $|X|\geq |H|$ and $\mathsf H(X)=H$ with $X$ a $k$-atom of $B$, it follows that $X$ is an $H$-coset, contrary to hypothesis. Therefore we instead conclude that $|H|<k$, whence $k'\geq 2$.  By \eqref{modatoms}, $\overline B$ is $k'$-separable in $G/H$ with $\overline X$ a $k'$-atom that is not a coset (since $k'\geq 2$ and $H=\mathsf H(X)$) and  $$\alpha_k(B)=|X|=|H+X|=|H||\overline X|=|H|\alpha_{k'}(\overline B).$$ Thus, since $\mathsf H(\overline X)$ is trivial with $k'\geq 2$, we can apply the argument of the previous paragraph to $\overline B$ using the $k'$-atom $\overline X$ to conclude that $\alpha_{k'}(\overline B)\leq k'-1+r'+\alpha_2(\overline X)$, where $r'=\frac{|X+B|}{|H|}-|\overline X|-|\overline B|$. Consequently, $r'\cdot |H|\leq |X+B|-|X|-|B|=r$ and \begin{align}\nn\alpha_k(B)&=|H|\alpha_{k'}(\overline B)\leq  \big(k'-1+r'+\alpha_2(\overline X)\big)|H|
=\lceil k/|H|\rceil\cdot |H|-|H|+r+|H|\alpha_2(\overline X),\end{align}
as desired.\end{proof}

\begin{corollary}\label{cor-atombounds}
Let $G=\Z/p\Z$ with $p\geq 2$ prime,  let $k\geq 2$, and let $B\subseteq G$ be a $k$-separable subset with $r\geq \kappa_k(B)-|B|$ and $|B|\geq 2$.
\begin{itemize}
\item[1.] $\alpha_k(B)\leq k+2r+2$
\item[2.] If $k\leq p-5r-6$, then $\alpha_k(B)\leq k+r+1+\sqrt{2r+2}$
\item[3.] If $k\leq p-5r-7$ and $r\geq 1$, then $\alpha_k(B)\leq k+r-\frac12+\sqrt{2r+\frac{17}{4}}$
\end{itemize}
\end{corollary}

\begin{proof}Since $B$ is $k$-separable with $k\geq 2$, it is also $2$-separable. Let $X$ be a $2$-atom of $B$, which cannot be a coset and must have $\mathsf H(X)$ trivial, both since $G=\Z/pZ$ has no proper, nontrivial subgroups. Thus  Theorem \ref{thm-katom-bound} yields \be\label{tickytack}\alpha_k(B)\leq k-1+r+\alpha_2(X).\ee Since $k\geq 2$ and $|B|\geq 2$, it follows that $X$ is $2$-separable with  $\kappa_2(X)-|X|\leq |X+B|-|B|-|X|=\kappa_2(B)-|B|\leq r$. Likewise, any $2$-atom of $X$ cannot be a coset. Thus Theorem \ref{thm-2atombound} implies $\alpha_2(X)\leq r+3$, which combined with \eqref{tickytack} yields Part 1. As a result, $|X|\leq k+2r+2$. Consequently, if $k\leq p-5r-6$, then $|X|\leq k+2r+2\leq p-3r-4$, in which case Theorem \ref{iso3}.1 implies that $\alpha_2(X)\leq 2+\sqrt{2r+2}$, which combined with \eqref{tickytack} yields Part 2.
Likewise, if $k\leq p-5r-7$ and $r\geq 1$, then $|X|\leq k+2r+2\leq p-3r-5$, in which case Theorem \ref{iso3}.2 implies that $\alpha_2(X)\leq \frac12+\sqrt{2r+\frac{17}{4}}$, which combined with \eqref{tickytack} yields Part 3.
\end{proof}

\subsection*{Results Regarding the $3k-4$ Theorem}

The following proposition shows how the conclusion of Conjecture \ref{conj-crit-pair} is equivalent to various statements regarding a sumset $X+Y$ being isomorphic to a sumset of subsets of the integers.

\begin{proposition}
\label{prop-3k4-equiv}
 Let $G=\Z/p\Z$ with $p\geq 2$ prime, let $\mathcal T=(A,B,C)$ be an additive trio from $G$ with $r(\mathcal T)=r$ and  $$|X|\geq r+3+\delta_X\quad\mbox{ for all $X\in \{A,B,C\}$},$$ where $\delta_X=1$ if $X\in \{A,B,C\}$ is a translate of one of the other two sets from $A$, $B$ and $C$, and otherwise $\delta_X=0$.  Then the following are equivalent.
\begin{itemize}
\item[1.] There is some nonzero $d\in G$ such that $\ell_d(X)\leq |X| +r+1$ for all $X\in \{A,B,C\}$.
\item[2.] There is some nonzero $d\in G$ such that $\ell_d(X)+\ell_d(Y)\leq p+1$ for every $(X,Y)\in \{(A,B), (B,C), (C,A)\}$.
\item[3.] There is some nonzero $d\in G$  such that $\ell_d(X)+\ell_d(Y)\leq p+1$ for some $(X,Y)\in \{(A,B), (B,C), (C,A)\}$.
\item[4.] The sumset $X+Y$ is  isomorphic to a sumset of subsets from $\Z$ for every $(X,Y)\in \{(A,B), (B,C), (C,A)\}$.
\item[5.] The sumset $X+Y$ is  isomorphic to a sumset of subsets from $\Z$ for some $(X,Y)\in \{(A,B), (B,C), (C,A)\}$.
\end{itemize}
In case Part 1 holds, then the $d$ for Part 2 may be taken to be the $d$ from  Part 1.
If Part 3 holds and either $\ell_d(X)\leq 2|X|-2$ or $d\in G$ is chosen such that  $\ell_d(X)+\ell_d(Y)$ is minimal, then Parts 1, 2 and 3 all hold using this $d$.
\end{proposition}

\begin{proof}
$1.\Rightarrow 2.$  Let $(X,Y)\in \{(A,B), (B,C), (C,A)\}$  be arbitrary and let $Z$ be remaining set from the trio.   Then Part 1 and the definition of $\mathsf r(\mathcal T)$ imply that  $\ell_d(X)+\ell_d(Y)\leq |X|+|Y|+2r+2=p-|Z|+r+2\leq p-1$, yielding 2 (with the same $d$). The implication $2.\Rightarrow 3.$ is trivial.  The implications $2.\Rightarrow 4.$ and $3.\Rightarrow 5.$ follow from the discussion in the Additive Isomorphism Subsection. The implication $4.\Rightarrow 5.$ is trivial. To show the equivalence of Parts 1 through 5, it remains to show $5.\Rightarrow 1.$

Suppose w.l.o.g. that $A+B\cong A'+B'$ for some $A',B'\subseteq \Z$.
Let $\delta=1$ if $A'$ and $B'$ are translates of each other and otherwise let $\delta=0$. As noted in the Additive Isomorphism Subsection, we have $\delta=1$ if and only if $A$ and $B$ are translates of each other, which is equivalent to $\delta_A=\delta_B=1$ by definition of $\delta_X$.
By an appropriate choice of translations (of $A$, $B$, $C$, $A'$ and $B'$), we can w.l.o.g. assume $0\in A\cap B$ and $0\in A'\cap B'$ with $0$ taken to $0$ by the isomorphism $\varphi:A'+B'\rightarrow A+B$, and that $0\notin A+B+C$, whence \eqref{inclusion} and our hypotheses imply \begin{align}\label{Cset}&C\subseteq -(A+B)^\mathsf c,\quad\und\\ \label{wintin}&|A+B|\leq p-|C|=|A|+|B|+r\leq |A|+|B|-3+\min\{|A|-\delta_A,|B|-\delta_B\}.\end{align} By replacing $\Z$ by $\la A'-A'+B'-B'\ra$, we can w.l.o.g. assume $\gcd(A'-A'+B'-B')=1$. Since $\delta=1$ if and only if $\delta_A=\delta_B=1$, \eqref{wintin} allows us to apply Theorem \cite[Theorem 20.2]{Grynk-book} to $A'+B'$ to conclude that the universal ambient group of $A'+B'$ is $\Z$. Thus, replacing $A'+B'$ by the isomorphic sumset given in the definition of a universal ambient group (this amounts to applying an appropriate automorphism of $\Z$ to  $A'$ and $B'$ and then translating as need be), it follows that the normalized Freiman isomorphism $\varphi:A'+B'\rightarrow A+B$ extends to a group homomorphism $\varphi :\Z\rightarrow G$, say with $\varphi(1)=d\in G$. Since $\delta=1$ if and only if $\delta_A=\delta_B=1$, \eqref{wintin} also allows us to apply Theorem \ref{thm-3k4Z} to $A'+B'$ to    find  arithmetic progressions $P'_{A}$, $P'_{B}$ and $P'_{A+B}$ with difference $1$ such that $A'\subseteq P'_{A}$ and $|P'_{A}|\leq |A|+r+1$, \ $B'\subseteq P'_B$ and $|P'_B|\leq |B|+r+1$, and $P'_{A+B}\subseteq A'+B'$ and $|P'_{A+B}|\geq |A|+|B|-1$.
 Let $P_X=\varphi(P'_X)$ for $X\in \{A,B,A+C\}$. Then $A\subseteq P_A$ and $B\subseteq P_B$ with $|P_A|\leq |A|+r+1$ and $|P_B|\leq |B|+r+1$ while $P_{A+B}\subseteq A+B$ with $|P_{A+B}|\geq |A|+|B|-1$. Set $P_C=-(P_{A+B})^\mathsf c$. Then $P_{A+B}\subseteq A+B$ implies $C\subseteq -(A+B)^\mathsf c\subseteq -(P_{A+B})^\mathsf c=P_C$ (in view of \eqref{Cset}) with $|P_C|=p-|P_{A+B}|\leq p+1-|A|-|B|=|C|+r+1$ (by definition of $\mathsf r(\mathcal T)$). The progressions $P_A$, \ $P_B$ and $P_C$ each have difference $\varphi(1)=d$ and show that $\ell_d(X)\leq |X|+r+1$ for every $X\in \{A,B,C\}$, yielding 1.

 Finally, it remains to show the implication $3.\Rightarrow 1.$ allows the $d$ from Part 3 to be used in Part 1, assuming the additional constraints. To this end, suppose w.l.o.g. that $\ell_d(A)+\ell_d(B)\leq p+1$. Let $P_A,\,P_B\subseteq G$ be minimal length arithmetic progressions with difference $d$ that contain $A$ and $B$, respectively. Per the discussion from the Additive Isomorphism Subsection, we can thus translate $A$ and $B$ (so that the initial terms in $P_A$ and $P_B$ are $0$) and find subsets $A',B'\subseteq [0,p-1]\subseteq \Z$ such that $A'+B'\cong A+B$ via the homomorphism $\varphi:\Z\rightarrow G$ given by $\varphi(1)=d\in G$. If $\ell_d(X)\leq 2|X|-2$, the Pigeonhole Principle ensures that there is a pair of consecutive terms from $P_A$ contained in $A$, in which case there is a pair of consecutive terms in $A'$ with difference $1$. Thus $\gcd(A'-A')=1$, implying $\gcd(A'-A'+B'-B')=1$. On the other hand, if $\gcd(A'-A'+B'-B')=n\geq 2$, then $\ell_{nd}(A)+\ell_{nd}(B)<\ell_d(A)+\ell_d(B)$ (since $|A|,|B|\geq r+3\geq 2$). As such, regardless of whether $d\in G$ has $\ell_d(X)\leq 2|X|-2$ or $\ell_d(A)+\ell_d(B)$ minimal, we conclude in both cases that $\gcd(A'-A'+B'-B')=1$. The proof of the implication $5.\Rightarrow 1.$ now shows that Part 1 holds using $d$.
\end{proof}

We will also need the following small $r$ version of Conjecture \ref{conj-crit-pair}, which follows by combining the main result from   \cite{lev-rect-threshold} with Theorem \ref{thm-3k4Z} and using  \cite[Theorem 20.2]{Grynk-book} to ensure the Freiman isomorphism given by  \cite{lev-rect-threshold} is well-behaved.

\begin{theorem}\label{thm-smallr}
Let $G=\Z/p\Z$ with $p\geq 2$ a prime, let $A,\,B\subseteq G$ be nonempty subsets with $A+B\neq G$. Set $C=-(A+B)^\mathsf c$. Suppose $|A|\geq |B|$, \  $|A\cup B|\leq \lceil\log_2p\rceil$ and  $$|A+B|=|A|+|B|+r\leq |A|+2|B|-3-\delta,$$ where $\delta=1$ for if $A$ and $B$ are translates of each other and otherwise $\delta=0$. Then there are arithmetic progressions $P_A,\,P_B,\,P_{C}\subseteq G$ all  having  common difference  such that $X\subseteq P_X$ and $|P_X|\leq |X|+r+1$ for all $X\in \{A,B,C\}$.
\end{theorem}

\begin{proof}
Since $|A\cup B|\leq \lceil\log_2|A|\rceil$, it follows from  \cite{lev-rect-threshold} (see also \cite[Theorem 20.4]{Grynk-book}) that $A+B$ is isomorphic to sumset of subsets of the integers, say $A+B\cong A'+B'$ with $A',\,B'\subseteq \Z$. Then, as in the proof of Proposition \ref{prop-3k4-equiv}, we can apply \cite[Theorem 20.2]{Grynk-book} to conclude that the universal ambient group of $A'+B'$ is $\Z$, allowing us to w.l.o.g. assume (translating sets appropriately) that there is a Freiman isomorphism $\varphi:A'+B'\rightarrow A+B$ that extends to a group homomorphism $\varphi:\Z\rightarrow G$, say with $\varphi(1)=d\in G$. Applying Theorem \ref{thm-3k4Z} to $A'+B'$ and applying $\varphi$ to the resulting progressions then yields the needed arithmetic progressions with difference $d$ using the same argument as for  Proposition \ref{prop-3k4-equiv}.
\end{proof}

The following is a very useful result of Huicochea \cite{huicochea-reduction-3k-4} that shows that, under some additional constraints, the weaker conclusion that one of the sets $X\in \{A,B,C\}$ is contained in a short progression  is also equivalent to the statements given in Proposition \ref{prop-3k4-equiv}. This means it is often sufficient to cover only one of the three sets in an additive trio by a short length arithmetic progression in order to establish the desired conclusion in Conjecture \ref{conj-crit-pair}.

\begin{theorem}\label{thm-mario-apred}
 Let $G=\Z/p\Z$ with $p\geq 2$ prime, let $(A,B,C)$ be an additive trio from $G$, and let $r,\,h\in\Z$ be integers.
\begin{itemize}
\item[1.] If $\ell_d(A)\leq |A|+h$ for some  $d\in G\setminus \{0\}$, \ $r(A,B,C)\leq r$, \  $|A|\geq r+3+h$ and $|B|\geq r+3+2h$ with strict inequality in one of these estimates,  and $|C|\geq r+3$, then  $\ell_d(A)\leq |A|+r+1$, \ $\ell_d(B)\leq |B|+r+1$ and $\ell_d(C)\leq |C|+r+1$.
\item[2.] If $\ell_d(A)\leq |A|+h$ for some $d\in G\setminus\{0\}$, \ $r(A,B,C)\leq h-1$, \ $h+2\leq |A|\leq \max\{|B|,|C|\}$, \ $|B|,|C|\geq h+3$ and $35h+10\leq p$, then $\ell_d(B)\leq |B|+h$ and $\ell_d(C)\leq |C|+h$.
\end{itemize}
\end{theorem}

\begin{proof}
1. The proof  is identical to that given for \cite[Theorem 1.2]{huicochea-reduction-3k-4}. One has only to use   Proposition \ref{prop-3k4-equiv} rather than \cite[Corollary 2.3]{huicochea-reduction-3k-4} and modify the calculations slightly to account for $h$ (the result in \cite{huicochea-reduction-3k-4} is only stated for $h=r+1$). Note in the derivation of (16) in \cite{huicochea-reduction-3k-4}, the cited equation (8) should be (11), and that Proposition \ref{prop-3k4-equiv} is not available for the case $h=0$, when $A$ is itself an arithmetic progression, if two of the  three sets $A$, $B$ and $C$ have size $r+3$ and are translates of each other. This case should be verified separately.  However,  if the desired conclusion fails when $A$ is an arithmetic progression, then either $|A+B|\geq |A|+|B|+r+1$ or $|A+C|\geq |A|+|C|+r+1$ can be derived (since $A+B\neq G$ and $A+C\neq G$), leading to a contradiction with $\mathsf r(A,B,C)\leq r$ via \eqref{inclusion}. Also, when using the implication $3.\Rightarrow 1.$ in Proposition \ref{prop-3k4-equiv}, the hypothesis $|A|\geq r+3+h$ ensures $\ell_d(A)\leq |A|+h\leq 2|A|-r-3\leq 2|A|-2$.

2. This is \cite[Theorem 1.3]{huicochea-reduction-3k-4} (with the notation altered via the substitution $h=r+1$).
\end{proof}



We conclude the section with the details for the construction from \cite{pablo-serra-3k4}.

\begin{proposition}\label{prop-pablo-eg-works}
Let  $n=4r+11$ with $r\geq 1$, let $G=\Z/n\Z$ and let   $A=[0,r+1]\cup \{r+3,2r+6\}\subseteq G$. Then the minimal length of an arithmetic progression containing $A$ is $2r+7$.
\end{proposition}

\begin{proof}
Note this is equivalent to showing that the maximal length of an arithmetic progression contained in \be\label{integers}A^\mathsf c=\{r+2\}\cup [r+4,2r+5]\cup [2r+7,4r+10]\ee is $2r+4=n-(2r+7)$. To this end,
let $P\subseteq A^\mathsf c$ be a maximal length arithmetic progression, say with difference $d\in [1,\frac{n-1}{2}]$. If $d=1$, then $|P|=\max\{1,r+2,2r+4\}=2r+4$ (since $r\geq -1$), as desired.

If $d\in [2,r+2]$, then $4r+10+d=n-1+d<n+r+2$, so any maximal length progression $P$ with difference $d$ contained in \eqref{integers} (considered as a subset of $\Z/n\Z$) must also  be such a maximal length progression contained in \eqref{integers} when considered instead as a subset of $\Z$. When $d=2$, to maximize its length, $P$ must then contain both $2r+5$ and $2r+7$, ensuring $4r+10\notin P$, and thus $P$ can be no larger than all odd numbers contained in $[r+2,4r+9]\subseteq \Z$. Thus  $|P|\leq \frac{3r+9}{2}\leq 2r+4$ (since $r\geq 1$), as desired.
If $d\in [3,r+2]$, then $|P|$ is no larger than the largest arithmetic progression with difference $d$ contained in $[r+2,4r+10]\subseteq \Z$. Thus  $|P|\leq \lceil\frac{3r+9}{d}\rceil\leq \lceil\frac{3r+9}{3}\rceil=r+3\leq 2r+4$ (since $r\geq -1$), as desired.

If $n=r+3=\frac{n+1}{4}$, then the argument for $d\leq r+2$ again yields the desired bound unless $4r+10\in P$. Assuming this is so, we have  $4r+10+d=r+2\in P$, followed by $r+2+d=2r+5\in P$, followed by $2r+5+d=3r+8\in P$. But then, since $3r+8+d=4r+11=0\notin A^\mathsf c$, we conclude that $3r+8$ is the final term in $P$. The term preceding $4r+10\in P$ must be $4r+10-d=3r+7\in P$, preceded by $3r+7-d=2r+4\in P$. Then, since $2r+4-d=r+1\notin A^\mathsf c$, it follows that $2r+4$ is the first term in $P$, whence $P=\{2r+4,3r+7,4r+10,r+2,2r+5,3r+8\}$ with $|P|=6\leq 2r+4$ (in view of $r\geq 1$), as desired.

If $d=\frac{n-1}{2}=2r+5$, then $P$ must be a translate of a set of the form $[1,\ell]\cup \big(\frac{n-1}{2}+[1,\ell]\big)$ or of the form $[1,\ell]\cup \big(\frac{n-1}{2}+[2,\ell]\big)$, where $\ell \geq 1$. In the former case, $A^\mathsf c$ must contain two disjoint intervals of length $\ell$, whence \eqref{integers} ensures $\ell\leq r+2$ and $|P|=2\ell\leq 2r+4$, as desired. In the later case,  $A^\mathsf c$ must contain two disjoint intervals, one of length $\ell$ and one of length $\ell-1$. Thus \eqref{integers} ensures $\ell-1\leq r+2$. If this inequality is strict, then $|P|=2\ell-1\leq 2r+3$, as desired. Otherwise, $\ell-1=r+2\geq 2$, which is only possible if $P=\big(\frac{n+1}{2}+[r+3,2r+5]\big)\cup [r+4,2r+5]=[r+4,2r+5]\cup [3r+9,4r+11]$. However this yields the contradiction $0\in P\subseteq A^\mathsf c$. It remains to consider $d\in [r+4,2r+4]$.

Let $$X=[0,r+1]\cup \{r+3\}.$$ Now $X$ is disjoint
from $A^\mathsf c$ and $(-2d+X)\cap (-d+X)=\emptyset$, the latter holding else $d\in X-X\subseteq n+[-(r+3),r+3]= [3r+8,n+r+3]$, which contradicts that $d\in [r+4,2r+4]$. Since $d\in [r+4,2r+4]$, we have
$-d+X\subseteq  [2r+7,4r+10]\subseteq A^\mathsf c$. Since $X$ is disjoint
from $A^\mathsf c$, it follows that any element from $-d+X$ contained in $P$ must be the last term in the progression $P$, ensuring at most one of the terms from $-d+X$ lies in $P$. As a result, $|P|\leq |A^\mathsf c|-|X|+1=2r+5$. If this inequality is strict, the desired bound on $|P|$ follows. Assuming otherwise, we conclude that $$P=A^\mathsf c\setminus (-d+X')$$ for some $X'\subseteq X$ with $|X'|=|X|-1=r+2$. Moreover, the unique term from the set $(-d+X)\setminus (-d+X')\subseteq [2r+7, 4r+10]$ must be the last term in $P$. This ensures that $-2d+X'$ is disjoint from $A^\mathsf c$.
Otherwise, any term $x\in (-2d+X')\cap A^\mathsf c$ must be contained in $P$ (since $P=A^\mathsf c\setminus (-d+X')$ with $-2d+X'$ and $-d+X'$ disjoint), and then $x+d\in P$ follows (as $x$ is not the last term in $P$), contradicting that $x+d\in -d+X'$ is disjoint from $P$ by definition. Since $r\geq 1$, it follows that either $-2d+X'\subseteq -2d+([0,r+1]\cup \{r+3\})$ contains two consecutive terms or else equals a three term arithmetic progression with difference $2$. Either way, it follows  from \eqref{integers}  that the only way $-2d+X'$ can be disjoint from $A^\mathsf c$ is if $-2d+X'\subseteq [0,r+3]$. However, since $X=[0,r+1]\cup \{r+3\}$ with $|X'|=|X|-1$, this is only possible if $2d\in [-2,1]\mod n$, contradicting  that $2d\in [2r+8,8r+8]\subseteq [2,n-3]$.
\end{proof}

\section{Density $1-\epsilon$}

A main contribution of \cite{pablo-grynk-3k-4-A=B} was an argument using Vosper duality to allow the \emph{combinatorial} portion of Freiman's original argument \cite{Freiman-vosp} \cite{Freiman-monograph} \cite{natbook} \cite{rodseth-2.4} to work at higher densities for $A+A$. Getting such arguments to also work for $A+B$, which we will accomplish in Proposition \ref{lemma-four-case1},  is the goal of the first part of this section. We begin with a simple consequence of the Cauchy-Davenport Theorem.

\begin{lemma}\label{lemma-cdtcor}
Let $p$ be a prime and let $X,\,Z\subseteq \Z/p\Z$ be subsets. Suppose $X$ and  $\bigcap_{x\in X}(x+Z)$ are nonempty and $Z\subset \Z/p\Z$ is proper. Then $$|\bigcap_{x\in X}(x+Z)|\leq |Z|-|X|+1.$$
\end{lemma}

\begin{proof}
Let $Y=\bigcap_{x\in X}(x+Z)$. Then $X$ and $Y$ are nonempty by hypothesis.
Observe that $Y-X\subseteq Z$, for if $x\in X$ and $y\in Y=\bigcap_{x\in X}(x+Z)$, then $y=x+z$ for some $z\in Z$, whence $y-x=x+z-x=z\in Z$. Thus, since $Z\subset \Z/p\Z$ is a proper subset and $X$ and $Y$ are nonempty, we can apply the Cauchy-Davenport Theorem (Theorem \ref{thm-cdt}) to $Y-X$ to conclude $|X|+|Y|-1\leq |Y-X|\leq |Z|$, yielding the desired upper bound for $|Y|=|\bigcap_{x\in X}(x+Z)|$.
\end{proof}

\begin{proposition}
\label{lemma-four-case1}
Let $p$ be prime, let $A,\,B \subseteq \Z/p\Z$ be nonempty subsets with  $$|A+B|=|A|+|B|+r\leq \frac34(p+1)\quad\und\quad p\geq 4r+9,$$ and set $C=-(A+B)^\mathsf c$. Suppose there exist subsets $A'\subseteq A$ and $B'\subseteq B$  or $A'\subseteq B$ and $B'\subseteq A$ such that  $A'+B'$ rectifies, $|B'|\leq |A'|$ and  \be\label{rectify-large-bound-sym-b}|A'|+2|B'|-4\geq |A+B|.\ee  Then there are arithmetic progressions $P_A,\,P_B,\,P_{C}\subseteq G$ all having  common difference  such that $X\subseteq P_X$ and $|P_X|\leq |X|+r+1$ for all $X\in \{A,B,C\}$
\end{proposition}

\begin{proof} We can w.l.o.g. assume $|B|\leq |A|$ and that $|A'|+|B'|$ is maximal among all subsets $A'$ and $B'$ satisfying the hypothesis of the proposition. Since $A'+B'$ rectifies, let $g\in \Z/p\Z$ be a nonzero difference with $\ell_g(A')+\ell_g(B')\leq p+1$  minimal.
 Let $$\{X,\,Y\}=\{A,\,B\}\quad\mbox{ with } \quad A'\subseteq X \und B'\subseteq Y.$$
Note $|B'|\leq |A'|$ and $A'\subseteq X$ ensure $|A'|\leq |X|\leq |A|$ and $|B'|\leq |A'|\leq |X|$, which combined with $B'\subseteq Y$ gives us $|B'|\leq \min\{|X|,|Y|\}=|B|$. In summary,  $$|A'|\leq |A|\;\und\;|B'|\leq |B|.$$
By hypothesis, $|X|\leq |A|\leq |A+B|\leq |A'|+2|B'|-4\leq |X|+2|B'|-4$, ensuring $|A'|\geq |B'|\geq 2$. Combined with the fact that $A'+B'$ rectifies, we find that $p+1\geq \ell_g(A')+\ell_g(B')\geq |A'|+|B'|\geq 4$, ensuring $p\geq 3$ is odd. If $p=3$, then, since  $A'+B'$ rectifies, it follows that $|A'|=|B'|=2$, whence $|A+B|\leq |A'|+2|B'|-4=2<p$. For $p>3$, we also have $|A+B|\leq \frac34(p+1)<p$. Thus $A+B\neq G$ in all cases.
Now $|A|+|B|+r=|A+B|\leq |A'|+2|B'|-4\leq |A|+2|B|-4$, ensuring \be\label{goup}|A|\geq |B|\geq r+4\quad\und\quad |A+B|=|A|+|B|+r\geq 3r+8.\ee
Set $$|A'+B'|=|A'|+|B'|+r'.$$
Translating $A$ and $B$ appropriately, we can assume
\be\label{ABC-intervals-sym} A'\subseteq [0,m]\cdot g, \quad B'\subseteq [0,n]\cdot g\quad\und\quad A'+B'\subseteq [0,m+n]\cdot g\ee with $0,\,mg\in A'$, \ $0,\,ng\in B'$, and  $m+n\leq p-1$. Thus $\ell_g(A')=m+1$ and $\ell_g(B')=n+1$.  Since $A'+B'$ rectifies, it follows that the map $\varphi:\Z/p\Z\rightarrow [0,p-1]\subseteq \Z$,
defined by $\varphi(sg)=s$ for $s\in [0,p-1]$, gives a Freiman isomorphism of $A'+B'$ with the integer sumset $\varphi(A')+\varphi(B')\subseteq \Z$. Observe that
\be\label{godat}{\gcd}\big(\varphi(A')-\varphi(A')+\varphi(B')-\varphi(B')\big)=1,\ee since if $\varphi(A')+\varphi(B')$ were contained in an arithmetic progression with difference $d\geq 2$, then this would also be the case for $\varphi(A')$ and $\varphi(B')$, and then $\ell_{dg}(A')+\ell_{dg}(B')<\ell_g(A')+\ell_g(B')$ would follow in view of $|A'|\geq |B'|\geq 2$, contradicting the minimality of $\ell_g(A')+\ell_g(B')$ for $g$.

In view of \eqref{rectify-large-bound-sym-b} and  $|B'|\leq |A'|$, we have $|A'+B'|\leq |A+B|\leq |A'|+|B'|+\min\{|A'|,\,|B'|\}-4$. Thus, in view of \eqref{godat}, we can apply the $3k-4$ Theorem (Theorem \ref{thm-3k4Z}) to the isomorphic sumset $\varphi(A')+\varphi(B')$. Then, letting $P_{A}=[0,m]\cdot g$, \ $P_{B}=[0,n]\cdot g$ and $P_{A+B}\subseteq A'+B'$ be the resulting arithmetic progressions with common difference $g$, we conclude that $|P_{A+B}|\geq |A'|+|B'|-1$, \be |P_A\setminus A'|\leq r'+1\quad \und  \quad |P_B\setminus B'|\leq r'+1.\label{rolling-a-sym-a}\ee  If $A'=X$ and $B'=Y$, then the original sumset $A+B=X+Y$ rectifies, $r'=r$ and the proposition follows with the progressions $P_A$, $P_B$ and $-(P_{A+B})^\mathsf c$ just defined (interchanging $P_A$ and $P_B$ if $X=B$). Therefore we can assume otherwise: \be\label{AorBnontempty-sym-a}\quad Y\setminus B'\neq \emptyset\quad\mbox{ or } \quad X\setminus A'\neq \emptyset.\ee

 Let $\Delta\geq 0$ be the integer such that   \be\nn |A+B|= |A'+B'|+\Delta,\ee in which case
 \be\label{r'-mainbound-sym} r'= |A\setminus A'|
 +|B\setminus B'|+r-\Delta.\ee
Since $|A'|+|B'|+r'=|A'+B'|=|A+B|-\Delta$,  it follows
 from
 \eqref{rectify-large-bound-sym-b} and $|B'|\leq |A'|$ that
 \be\label{1st-r'} r'\leq |B'|-4-\Delta\quad\und\quad r'\leq |A'|-4-\Delta.\ee
 Averaging both bounds above  together with the bound \eqref{r'-mainbound-sym}, and recalling that $|A+B|=|A|+|B|+r$, we obtain  \be\label{r'-special-sym} r'\leq \frac13|A+B|-\frac83-\Delta.\ee

\subsection*{Step A} $|-A'+(A'+Y)^\mathsf c|\leq |(A'+Y)^\mathsf c|+2|A'|-4$.

\begin{proof}
If Step A fails, then it follows from  $p-|A+B|=|(A+B)^\mathsf c|=|(X+Y)^\mathsf c|\leq |(A'+Y)^\mathsf c|$ and Proposition \ref{lemma-dual} that  $$p-|A+B|+2|A'|-3\leq |(A'+Y)^\mathsf c|+2|A'|-3\leq |-A'+(A'+Y)^\mathsf c|\leq |Y^\mathsf c|=p-|Y|.$$ Thus \eqref{rectify-large-bound-sym-b}, $|B'|\leq |Y|$ and $|B'|\leq |A'|$ imply
$$|Y|+2|A'|-3\leq |A+B|\leq |A'|+2|B'|-4\leq |Y|+2|A'|-4,$$ which is not possible.
\end{proof}

\subsection*{Step B} $|-B'+(X+B')^\mathsf c|\leq |(X+B')^\mathsf c|+2|B'|-4$.

\begin{proof}
If Step B fails, then from $p-|A+B|=|(A+B)^\mathsf c|=|(X+Y)^\mathsf c|\leq |(X+B')^\mathsf c|$ and Proposition  \ref{lemma-dual} it follows that  $$p-|A+B|+2|B'|-3\leq |(X+B')^\mathsf c|+2|B'|-3\leq |-B'+(X+B')^\mathsf c|\leq |X^\mathsf c|=p-|X|.$$ Thus \eqref{rectify-large-bound-sym-b} and $|A'|\leq |X|$ imply
$$|X|+2|B'|-3\leq |A+B|\leq |A'|+2|B'|-4\leq |X|+2|B'|-4,$$ which is not possible.
\end{proof}

\subsection*{Step C} If $|A'|+|Y|\geq |B'|+|X|$, then $|-A'+(A'+Y)^\mathsf c|\leq |A'|+2|(A'+Y)^\mathsf c|-3$.

\begin{proof}  Note $2p-2|A+B|=2|(A+B)^\mathsf c|=2|(X+Y)^\mathsf c|\leq 2|(A'+Y)^\mathsf c|$, while Proposition \ref{lemma-dual} implies $|-A'+(A'+Y)^\mathsf c|\leq |Y^\mathsf c|$. As a result, if Step C fails, we obtain  $$|A'|+2p-2|A+B|-2\leq |A'|+2|(A'+Y)^\mathsf c|-2\leq |-A'+(A'+Y)^\mathsf c|\leq |Y^\mathsf c|=p-|Y|.$$ Collecting terms in the above inequality, multiplying by $2$, and using the hypothesis of Step C, we obtain
\be 2p\leq 4|A+B|-2|Y|-2|A'|+4\leq 4|A+B|-|X|-|Y|-|A'|-|B'|+4.\nn\ee
Using the identities $|X|+|Y|=|A|+|B|=|A+B|-r$ and $|A'|+|B'|=|A'+B'|-r'=|A+B|-r'-\Delta$ in the above estimate, and then applying the estimate \eqref{r'-special-sym},  yields \be\label{larkm}2p\leq 2|A+B|+r+r'+\Delta+4
\leq \frac{7}{3}|A+B|+r+\frac43.\ee

By hypothesis, $|A+B|\leq \frac{3p+3}{4}$ and $p\geq 4r+9$. Since $p\geq 3$  is odd, strict inequality in the first estimate together with \eqref{larkm} yields
$\frac{3p+1}{4}\geq |A+B|\geq \frac{6p-3r-4}{7}$, implying $p\leq 4r+7$, which is contrary to hypothesis. Therefore we can instead assume  equality holds:  $|A+B|= \frac{3p+3}{4}$. As a result, $p\equiv 3\mod 4$, in which case the hypothesis $p\geq 4r+9$ improves to $p\geq 4r+11$. In such case, \eqref{larkm} yields $\frac{3p+3}{4}= |A+B|\geq \frac{6p-3r-4}{7}$, implying $p\leq 4r+11$ (recall $p\equiv 3\mod 4$). Thus $$p=4r+11\quad\und\quad |A+B|=\frac34(p+1)=3(r+3)=3r+9.$$
Since \eqref{goup} implies $3r+9=|A+B|=|A|+|B|+r$ with $|A|\geq |B|\geq r+4$, it follows that   $|B|=r+4$, $|A|=r+5$ and $|A|+2|B|-4=|A+B|$. Hence, since $|A+B|\leq |A'|+2|B'|-4\leq |A|+2|B|-4$, it follows that  $|A'|=|A|$ and $|B'|=|B|$. But then  $A'=X$ and $B'=Y$, contradicting \eqref{AorBnontempty-sym-a}.
\end{proof}

\subsection*{Step D} If $|B'|+|X|\geq |A'|+|Y|$, then $|-B'+(X+B')^\mathsf c|\leq |B'|+2|(X+B')^\mathsf c|-3$.

\begin{proof}  Note $2p-2|A+B|=2|(A+B)^\mathsf c|=2|(X+Y)^\mathsf c|\leq 2|(X+B')^\mathsf c|$, while Proposition \ref{lemma-dual} implies $|-B'+(X+B')^\mathsf c|\leq |X^\mathsf c|$. As a result, if Step D fails, we obtain
$$|B'|+2p-2|A+B|-2\leq |B'|+2|(X+B')^\mathsf c|-2\leq |-B'+(X+B')^\mathsf c|\leq |X^\mathsf c|=p-|X|.$$ Collecting terms in the above inequality, multiplying by $2$, and using the hypothesis of Step D, we obtain
\be 2p\leq 4|A+B|-2|X|-2|B'|+4\leq 4|A+B|-|X|-|Y|-|A'|-|B'|+4,\nn\ee
and we now obtain the same contradiction as at the end of Step C.
\end{proof}

By our application of the $3k-4$ Theorem (Theorem \ref{thm-3k4Z}) to $\varphi(A')+\varphi(B')$, we know that $A'+B'$ contains an arithmetic progression $P_{A+B}$ with difference $g$ and length $|P_{A+B}|\geq |A'|+|B'|-1$, which implies \be\label{A'B'short}\ell_g\big((A'+B')^\mathsf c\big)\leq p-|A'|-|B'|+1.\ee  From \eqref{rolling-a-sym-a} and  \eqref{1st-r'}, we obtain  \be\label{A'length}\ell_g(-A')=\ell_g(A')\leq |P_A|\leq |A'|+r'+1\leq |A'|+|B'|-3.\ee Combined with \eqref{A'B'short}, it follows that  $\ell_g(-A')+\ell_g\big((A'+B')^\mathsf c\big)\leq p-2$, ensuring $-A'+(A'+B')^\mathsf c$ rectifies via the difference $g$. Since $(A'+Y)^\mathsf c\subseteq (A'+B')^\mathsf c$, it follows that $-A'+(A'+Y)^\mathsf c$ also rectifies via the difference $g$.

From \eqref{rolling-a-sym-a} and  \eqref{1st-r'}, we obtain  \be\label{B'length}\ell_g(-B')=\ell_g(B')\leq |P_B|\leq |B'|+r'+1\leq |A'|+|B'|-3.\ee Combined with \eqref{A'B'short}, it follows that  $\ell_g(-B')+\ell_g\big((A'+B')^\mathsf c\big)\leq p-2$, ensuring $-B'+(A'+B')^\mathsf c$ rectifies via the difference $g$. Since $(X+B')^\mathsf c\subseteq (A'+B')^\mathsf c$, it follows that $-B'+(X+B')^\mathsf c$ also rectifies via the difference $g$.

By our application of the $3k-4$ Theorem (Theorem \ref{thm-3k4Z}) to $\varphi(A')+\varphi(B')$, we know $\varphi(A')$ is contained in the arithmetic progression $\varphi(P_A)=[0,m]$ with difference $1$ and length $|P_A|\leq |A'|+r'+1$, with the latter inequality by \eqref{rolling-a-sym-a}. Moreover, \eqref{1st-r'} ensures that  $r'+1\leq |B'|-3\leq |A'|-3$, in which case  $\varphi(A')$ contains at least $2$ consecutive elements. Hence \be\label{gcd-a}{\gcd}\big(\varphi(A')-\varphi(A')\big)=1.\ee
Likewise, by our application of the $3k-4$ Theorem (Theorem \ref{thm-3k4Z}) to $\varphi(A')+\varphi(B')$, we know $\varphi(B')$ is contained in the arithmetic progression $\varphi(P_B)=[0,n]$ with difference $1$ and length $|P_B|\leq |B'|+r'+1$, with the latter inequality by \eqref{rolling-a-sym-a}. Moreover, \eqref{1st-r'} ensures that  $r'+1\leq |B'|-3$, in which case  $\varphi(B')$ contains at least $2$ consecutive elements. Hence \be\label{gcd-b}{\gcd}\big(\varphi(B')-\varphi(B')\big)=1.\ee
We divide the remainder of the proof  into two short, and nearly identical, subcases.

\subsection*{CASE 1:} $|A'|+|Y|>|B'|+|X|$.

\smallskip

Since $-A'+(A'+Y)^\mathsf c$ rectifies via the difference $g$, we see that  it is isomorphic to the sumset $\varphi(m-A')+\varphi(-x+(A'+Y)^\mathsf c)$, where $x\in \Z/p\Z$ is the first term in the minimal length arithmetic progression with difference $g$ containing $(A'+Y)^\mathsf c$. Hence, in view of  \eqref{gcd-a}, Step A, the case hypothesis and Step C, we can apply the $3k-4$ Theorem (Theorem \ref{thm-3k4Z}) to the isomorphic sumset $\varphi(m-A')+\varphi(x+(A'+Y)^\mathsf c)$ and thereby conclude that there is an arithmetic progression $P\subseteq -A'+(A'+Y)^\mathsf c$ with difference $g$ and length $|P|\geq |A'|+|(A'+Y)^\mathsf c|-1\geq |A'|+|(X+Y)^\mathsf c|-1=p-|A+B|+|A'|-1$. Consequently, since Proposition \ref{lemma-dual} ensures that $P\subseteq -A'+(A'+Y)^\mathsf c\subseteq Y^\mathsf c$, it follows that $$\ell_g(Y)\leq |A+B|-|A'|+1.$$ Combined with \eqref{A'length}, we find that
\be\label{rectifygo}\ell_g(A')+\ell_g(Y)\leq |A+B|+r'+2.\ee

If  $A'+Y$ does not rectify,  then \eqref{rectifygo} and \eqref{r'-special-sym} imply $p\leq |A+B|+r'\leq \frac43|A+B|-\frac83$, whence $$|A+B|\geq \frac34p+2,$$ contrary to hypothesis.  Therefore  we instead conclude that $A'+Y$ rectifies. Consequently, $A'+Y'$ also rectifies for any subset $Y'$ with $B'\subseteq Y'\subseteq Y$. As a result, if  $|B'|<|Y|$, then there is subset $Y'$ with $B\subseteq Y'\subseteq Y$ and $|Y'|=|B'|+1$ such that $A'+Y'$ rectifies. If $|B'|<|A'|$, then $|Y'|=|B'|+1\leq |A'|$ and $|A'|+2|Y'|-4\geq |A'|+2|B'|-2\geq |A+B|$, which would contradict the maximality of $|A'|+|B'|$. If $|B'|=|A'|$, then $|Y'|=|B'|+1> |A'|$ and $|Y'|+2|A'|-4=|A'|+2|B'|-3\geq |A+B|$, also contradicting the maximality of $|A'|+|B'|$.
We are left to conclude that $|B'|=|Y|$. But now  the case hypothesis yields  $|A'|+|Y|>|B'|+|X|=|Y|+|X|$, yielding the contradiction $|A'|>|X|$ (as $A'\subseteq X$).

\subsection*{CASE 2:} $|B'|+|X|\geq |A'|+|Y|$.

\smallskip

This subcase is nearly identical with the last, though there are a few small differences.
Since $-B'+(X+B')^\mathsf c$ rectifies via the difference $g$, it is isomorphic to $\varphi(n-B')+\varphi(-y+(X+B')^\mathsf c)$, where $y\in \Z/p\Z$ is the first term in the minimal length arithmetic progression with difference $g$ containing $(X+B')^\mathsf c$. Hence, in view of  \eqref{gcd-b}, Step B, the case hypothesis and Step D, we can apply the $3k-4$ Theorem (Theorem \ref{thm-3k4Z}) to the isomorphic sumset $\varphi(n-B')+\varphi(y+(X+B')^\mathsf c)$ and thereby conclude that there is an arithmetic progression $P\subseteq -B'+(X+B')^\mathsf c$ with difference $g$ and length $|P|\geq |B'|+|(X+B')^\mathsf c|-1\geq |B'|+|(X+Y)^\mathsf c|-1=p-|A+B|+|B'|-1$. Consequently, since Proposition \ref{lemma-dual} ensures that $P\subseteq -B'+(X+B')^\mathsf c\subseteq X^\mathsf c$, it follows that $$\ell_g(X)\leq |A+B|-|B'|+1.$$ Combined with \eqref{B'length}, we find that
\be\label{rectifygoII}\ell_g(B')+\ell_g(X)\leq |A+B|+r'+2.\ee

If  $X+B'$ does not rectify,  then \eqref{rectifygoII} and \eqref{r'-special-sym} imply $p\leq |A+B|+r'\leq \frac43|A+B|-\frac83$, whence $|A+B|\geq \frac34p+2$, contrary to hypothesis.  Therefore  we instead conclude that $X+B'$ rectifies, in which case $X+B'$ contradicts the maximality of $|A'|+|B'|$ unless $A'=X$. However, $A'=X$ implies $|A'|=|X|$, which combined with the case hypothesis yields $|B'|\geq |Y|$. Since $B'\subseteq Y$, this in turn forces $B'=Y$. But now $A'=X$ and $B'=Y$, contradicting \eqref{AorBnontempty-sym-a}, which completes CASE 2, and the proof as well.
\end{proof}

With Proposition  \ref{lemma-four-case1} in hand, the method of Freiman \cite{Freiman-vosp} \cite{Freiman-monograph} \cite{natbook} \cite{rodseth-2.4} (used for $A+A$) can now be pushed to higher densities for $A+B$, as was done in \cite{pablo-grynk-3k-4-A=B} for $A+A$. However, to deal with general sumsets $A+B$ rather than $A+A$, we will need to use a new variation for how the exponential sums are estimated in the argument. For following the numerical calculations, both here and later, the reader may wish to use an algebraic computation system since they can be very tedious to do by hand.

\begin{theorem}\label{fouier-II-distinct}
Let $G=\Z/p\Z$ with $p\geq 2$ a prime, let $\beta\in [0.731,1]$ and $\alpha\in (0, 0.212]$ be real numbers, let $A,\, B\subseteq G$ be nonempty subsets, and set $C=-(A+B)^\mathsf c$. Suppose   $$|A+B|=|A|+|B|+r\leq |A|+(1+\alpha)|B|-3\quad\und\quad \beta|A|\leq |B|\leq |A|\leq \min\{c_1,c_\beta\}p,$$ where
$c_1=\frac{5-18\alpha-24\alpha^2-8\alpha^3}{14-9\alpha-24\alpha^2-8\alpha^3}$ and $c_{\beta}=\frac{-4-(1+12\alpha)\beta+3(5+\alpha-4\alpha^2)\beta^{2}+
(-1+3\alpha-4\alpha^3)\beta^{3}-4(1+\alpha)^3\beta^{4}}{-4+11\beta^{2}+11\beta^{3}
-4\beta^{5}-3\alpha\beta
(4-5\beta^{2}+4\beta^{4})-12\alpha^2(\beta^{2}+\beta^{5})-4\alpha^3(\beta^{3}
+\beta^{5})}.
$
Then there are arithmetic progressions $P_A,\,P_B,\,P_{C}\subseteq G$ all having  common difference  such that $X\subseteq P_X$ and $|P_X|\leq |X|+r+1$ for all $X\in \{A,B,C\}$.\end{theorem}

\begin{proof}
The real roots/singularities of $c_1$ are $\alpha\cong 0.2129$, $-2$, $-1.56$ and $0.56$. Thus $c_1\in (0,\frac{5}{14}]$ for all $\alpha\in (0,0.212]$. Consider the denominator of $c_\beta$: $$-4+11\beta^{2}+11\beta^{3}
-4\beta^{5}-3\alpha\beta
(4-5\beta^{2}+4\beta^{4})-12\alpha^2(\beta^{2}+\beta^{5})-4\alpha^3(\beta^{3}
+\beta^{5}).$$
Calculating successive derivatives with respect to $\alpha$, it follows that all partial derivatives with respect to $\alpha$ are non-positive, for  $\alpha\geq 0$ and $\beta>0$. Hence, since $\alpha\leq 0.212$, it follows that the denominator of $c_\beta$ is at least the value obtained under the substitution $\alpha=0.212$, which is a $5$-th degree polynomial whose only real roots are $\beta\approx -0.825$, $0.562$ and $1.6199$, from which we can derive that the denominator of $c_\beta$ is positive for all $\alpha\in (0,0.212]$ and $\beta\in [0.731,1]$. Thus the hypothesis $0<|A|\leq c_\beta p$ ensures that both  $c_\beta$ and the numerator of $c_\beta$ are positive.
We have  $|A+B|<(2+\alpha)|A|\leq (2+\alpha)c_1p
=\frac{(5-18\alpha-24\alpha^2-8\alpha^3)(2+\alpha)}{14-9\alpha-24\alpha^2-8\alpha^3}p$. The derivative of the coefficient of $p$ equals $\frac{2(-43-128\alpha-60\alpha^2+64\alpha^3+32\alpha^4)}{(-7+8\alpha+8\alpha^2)^2}$, which has roots/discontinuities for $\alpha \approx -1.56$, $0.56$,  $-2.139$, $-0.8455$, $-0.5$ and  $1.48$ with the derivative negative for $\alpha\in [0,0.5]$. Thus $|A+B|\leq \frac{(5-18\alpha-24\alpha^2-8\alpha^3)(2+\alpha)}
{14-9\alpha-24\alpha^2-8\alpha^3}p\leq \frac{10}{14}p<0.715\, p.$
As a result, \be|A+B|\leq \frac34p\label{hite}.\ee
In particular,
$A+B\neq G$, so  the Cauchy-Davenport Theorem (Theorem \ref{thm-cdt})  implies $r\geq -1$.
By hypothesis, $\alpha<1$ and $$r\leq \alpha|B|-3<|B|-3.$$
Thus $|A|\geq |B|\geq r+4$, whence $3r+8\leq |A|+|B|+r=|A+B|\leq  \frac34p$, in turn implying that
\be\label{r-small} p\geq 4r+11\quad\und\quad r\leq \frac14 p-2.75.\ee
Let $$\ell=|A+B|=|A|+|B|+r, \quad a=|A| \quad \und \quad b=|B|.$$
In view of \eqref{hite} and \eqref{r-small}, we can apply Proposition \ref{lemma-four-case1} and thereby complete the proof unless every  pair of subsets $A'\subseteq A$  and  $B'\subseteq B$ such that    $A'+B'$ rectifies has   \be\nn |A'|+|B'|+\min\{|A'|,\,|B'|\}\leq |A+B|+3=\ell+3,\ee
in turn implying   \be\label{assump-!!}\min\{|A'|,\,|B'|\}\leq\frac13\ell+1.\ee


For the rest of this proof, let us identify $\Z/p\Z$ with the set of integers $[0,p-1]$ with addition mod $p$. Then, for every $X\subseteq \Z/p\Z$ and $d\in \Z/p\Z$, we have $S_X(d)=\sum_{x\in X}\exp(dx/p) \in \C$.

\smallskip

\textbf{Claim A:} For  every  $d\in [1,p-1]$, either $|S_A(d)|\leq \gamma_A:=\frac{2\ell+6-3|A|}{3}\leq \frac13\big(2 \alpha b+2b-a\big)$ or $|S_B(d)|\leq \gamma_B:=\frac{2\ell+6-3|B|}{3}\leq \frac13\big(2\alpha b+2a-b\big)$.

\begin{proof}Let $d\in [1,p-1]$ be arbitrary.
For $u,\,v\in [0,1)$, consider the open half-arcs $C_u=\{\exp(x):\; x\in (u,u+\frac12)\}$ and $C_v=\{\exp(x):\; x\in (v,v+\frac12)\}$  in the unit circle in $\C$. Let $A'=\{x\in A:\;\exp(dx/p)\in C_u\}$ and $B'=\{x\in B:\;\exp(dx/p)\in C_v\}$, and choose $u,\,v\in [0,1)$ to maximize $|A'|$ and $|B'|$.  Since the set of $p$-th roots of unity contained in $C_u$ or $C_v$ correspond to an arithmetic progression of difference 1 in $\Z/p\Z$, it is clear that, for $d^*$ the multiplicative inverse of $d$ modulo $p$, we have $\ell_{d^*}(A')\leq \frac{p+1}{2}$ and $\ell_{d^*}(B')\leq \frac{p+1}{2}$. Hence the sumset $A'+B'$ rectifies.
 Then \eqref{assump-!!} implies that $\min\{|A'|,\,|B'|\}\leq \frac13\ell+1$. If $|B'|\leq |A'|$, then the maximality of $v$ together with Lemma \ref{lem-freiman-circle} implies $|S_B(d)|\leq 2(\frac13\ell+1)-|B|$. If $|A'|\leq |B'|$, then the maximality of $u$ together with Lemma \ref{lem-freiman-circle} implies $|S_A(d)|\leq 2(\frac13\ell+1)-|A|$. In either case, one of the desired bounds follows.
\end{proof}

In view of Claim A, we can partition $[1,p-1]=I_A\cup I_B$ such that $|S_A(d)|\leq \gamma_A= \frac{2\ell+6-3|A|}{3}$ for $d\in I_A$, and $|S_B(d)|\leq \gamma_B=\frac{2\ell+6-3|B|}{3}$ for $d\in I_B$. Moreover, set $C_d=B$ and $D_d=A$ if $d\in I_A$, and $C_d=A$ and $D_d=B$ for $d\in I_B$.
 To complete the proof, we now exploit Claim A to obtain a contradiction, using in particular the following manipulations which are a new variation on  fourier estimates used in Additive Combinatorics (e.g. \cite[Corollary 19.1]{Grynk-book} \cite[pp. 290--291]{Grynk-book}).
By Fourier inversion, the triangle inequality, the fact that $S_A(0)=|A|=a$, $S_B(0)=|B|=b$ and $S_{A+B}(0)=\ell$, and the Cauchy-Schwarz Inequality, we have
\begin{align}\nn
abp= \Sum{x=0}{p-1}S_A(x)S_B(x)\overline{S_{A+B}(x)}
=ab\ell+\Sum{x=1}{p-1}S_A(x)S_B(x)\overline{S_{A+B}(x)}\\\nn
\leq ab\ell+\Sum{x=1}{p-1}|S_A(x)|\, |S_B(x)|\,|S_{A+B}(x)|
\leq ab\ell+\Sum{x=1}{p-1}\gamma_{D_x}|S_{C_x}(x)|\,|S_{A+B}(x)|
\\\nn\leq
ab\ell+\Big(\Sum{x=1}{p-1}
\gamma_{D_x}^2|S_{C_x}(x)|^2\Big)^{1/2}\Big(\Sum{x=1}{p-1}|S_{A+B}(x)|^2\Big)^{1/2}\\\nn
\leq
ab\ell+\Big(\gamma_{B}^2\Sum{x=1}{p-1}
|S_{A}(x)|^2+\gamma_A^2\Sum{x=1}{p-1}
|S_B(x)|^2\Big)^{1/2}\Big(\Sum{x=1}{p-1}|S_{A+B}(x)|^2\Big)^{1/2}
\\
=ab\ell+\big(\gamma_B^2(ap-a^2)+\gamma_A^2(bp-b^2)\big)^{1/2}(\ell p-\ell^2)^{1/2}.\nn
\end{align}
Thus $a^2b^2(p-\ell)^2\leq \big(\gamma_B^2a(p-a)+\gamma_A^2b(p-b)\big)\ell(p-\ell),$
in turn implying
\be\label{starter}a^2b^2(p-\ell)-\big(\gamma_B^2a(p-a)+\gamma_A^2b(p-b)\big)\ell\leq 0.\ee
It remains to derive an estimate for $p/a$ from \eqref{starter} that is contrary to our hypotheses.

Let $$x=\frac{b}{a}\;\und\;z=\frac{p}{a}.$$ Note $$0.731\leq \beta\leq x\leq 1\;\und\;z\geq \max\{c_1^{-1},c_\beta^{-1}\}.$$
Multiplying \eqref{starter} by $9/a^5$, applying the estimate $\ell=|A+B|<|A|+(1+\alpha)|B|=a+(1+\alpha)b$, and  using the estimates for $\gamma_A$ and $\gamma_B$ given in Step A, we obtain
\be\label{starstuff-a} 9x^2(z-1-x-\alpha x)-(1+x+\alpha x)\Big((2\alpha x+2-x)^2(z-1)
+(2\alpha x+2x-1)^2x(z-x)\Big)<0.\ee

Let $f(x,z,\alpha)$ denote the left hand side of \eqref{starstuff-a}. Now
\begin{align*}
&\frac{\partial^2}{\partial x^2}f=(30-6x-48x^2+6\alpha-24\alpha^2+18\alpha x-24\alpha^3x-144\alpha x^2-144\alpha^2x^2-48\alpha^3x^2)z\\&\quad\quad\quad-22-66x+80x^3+24\alpha^2-90\alpha x+24 \alpha^3 x+240\alpha x^3+240\alpha^2x^3+80\alpha^3x^3.\end{align*}
Since $0.731\leq\beta\leq x\leq 1$ and $\alpha>0$, we have $30-6x-48x^2<0$ and $6\alpha +18\alpha x-144\alpha x^2\leq 0$, so  the coefficient of $z$ above is negative Thus, using the hypothesis $z\geq c_1^{-1}\geq \frac{14}{5}=2.8$, we find that
\begin{align*}\frac{\partial^2}{\partial x^2}f\leq 62-82.8x-134.4x^2+80x^3+\alpha(16.8-39.6x-403.2x^2+240x^3)\\+\alpha^2(-43.2-403.2x^2+240x^3)
+\alpha^3(-43.2x-134.4x^2+80x^3)\end{align*} Since $0.731\leq\beta\leq x\leq 1$, it follows that each coefficient of a term $\alpha^i$ above (for $i\geq 1$) is negative, whence applying the estimate $\alpha\geq 0$ yields $\frac{\partial^2}{\partial x^2}f\leq 62-82.8x-134.4x^2+80x^3<0$, with latter inequality in view of $0.5<x<2$. Thus $\frac{\partial^2}{\partial x^2}f<0$,
ensuring that the value of $f(x,z,\alpha)$ will be minimized at an extremal value for $x$, and thus \eqref{starstuff-a} holds either with   $x=1$ or $x=\beta\geq 0.731$.

Suppose \eqref{starstuff-a} hold with $x=1$. Then \eqref{starstuff-a} yields
$$-14+9\alpha+24\alpha^2+8\alpha^3+(5-18\alpha-24\alpha^2-8\alpha^3)z< 0.$$
Since $\alpha\leq 0.212$, it follows that the coefficient of $z$ above is positive, whence $z< \frac{14-9\alpha-24\alpha^2-8\alpha^3}{5-18\alpha-24\alpha^2-8\alpha^3}=c_1^{-1}$, contrary to hypothesis. So we may now instead assume \eqref{starstuff-a} holds with $x=\beta\geq 0.731$. In this case, \eqref{starstuff-a} instead implies
$z<c_\beta^{-1}$, again contrary to hypothesis, completing the proof.
 \end{proof}

To gain a feel for possible values of the constants allowed in Theorem \ref{fouier-II-distinct}, we state three explicit  cases. If one compares with the recent work of Huicochea \cite{huicochea-3k-4-distinct} mentioned in the introduction, here we obtain better constants at the expense of $|A|$ and $|B|$ being required to be  much closer in size. Since we will later be able to remove such restrictions using isoperimetric arguments, these added restrictions are  less of a concern with Theorem \ref{fouier-II-distinct} an intermediary step for us.

\begin{corollary}\label{cor-fouier-II-distinct}
Let $G=\Z/p\Z$ with $p\geq 2$ a prime, let $A,\, B\subseteq G$ be nonempty subsets, and set $C=-(A+B)^\mathsf c$. Suppose   either \begin{align*}&|A+B|=|A|+|B|+r\leq |A|+1.021|B|-3\quad\und\quad \frac45|A|\leq |B|\leq |A|\leq \frac{p}{3}, \quad \mbox{ or}\\
&|A+B|=|A|+|B|+r\leq |A|+1.105|B|-3\quad\und\quad \frac45|A|\leq |B|\leq |A|\leq \frac{p}{5},\quad{ or}\\
&|A+B|=|A|+|B|+r\leq |A|+(1+\frac19)|B|-3\quad\und\quad 0.8484|A|\leq |B|\leq |A|\leq \frac{1963}{9253}\,p
.\end{align*}
Then there are arithmetic progressions $P_A,\,P_B,\,P_{C}\subseteq G$ all having  common difference  such that $X\subseteq P_X$ and $|P_X|\leq |X|+r+1$ for all $X\in \{A,B,C\}$.\end{corollary}

\begin{proof}
This follows from Theorem \ref{fouier-II-distinct} applied with $\alpha=0.021$ and  $\beta=.8$, applied with $\alpha=0.105$ and $\beta=.8$, and applied with $\alpha=\frac19$ and $\beta=0.8484$
\end{proof}

Our next step is to use isoperimetric machinery to remove most size and density restrictions from Theorem \ref{fouier-II-distinct}, resulting in the following theorem.

\begin{theorem}
\label{thm-r-O(p)-a}
Let $G=\Z/p\Z$ with $p\geq 2$ a prime, let $A,\, B\subseteq G$ be nonempty subsets with $A+B\neq G$, and set $C=-(A+B)^\mathsf c$. Suppose
  $|A|\geq |B|$, \ $\frac{9253}{1963}\Big(10r+26.5+\sqrt{2r+\frac{17}{4}}\Big)\leq p$, and  \be\nn|A+B|=|A|+|B|+r\leq \min\{|A|+\big(1+\frac19\big)|B|-3,\;p-9(r+3)\}.\ee
Then there are arithmetic progressions $P_A,\,P_B,\,P_{C}\subseteq G$ all having  common difference  such that
\begin{align*} &X\subseteq P_X\;\und\; |P_X|\leq |X|+r+1\quad \mbox{ for all $X\in \{A,B,C\}$}.\end{align*}
\end{theorem}

\begin{proof}
Since $A+B\neq G$, the Cauchy-Davenport Theorem (Theorem \ref{thm-cdt}) ensures that $r\geq -1$, whence $|A+B|\leq p-9(r+3)\leq p-18$.
By hypothesis, we have $$|A|,|B|,|C|\geq 9(r+3)\geq r+19\geq 18.$$ As a result, if $r\leq 0$, then Conjecture \ref{conj-crit-pair} holds for $A+B$ (as mentioned in the Small $r$ Subsection of the  Introduction), and the desired conclusion follows. Therefore we can assume $r\geq 1$. By hypothesis, we have \be\label{p-bigwithr} p\geq \frac{9253}{1963}\Big(10r+26.5+\sqrt{2r+\frac{17}{4}}\Big)\geq 14r+34.\ee

Since $|A|,|C|\geq 9(r+3)$, it follows that $B$ is $9(r+3)$-separable with $\kappa_{9(r+3)}(B)\leq |A+B|-|A|=|B|+r$. Let $X$ be a $9(r+3)$-atom of $B$. Then \be\label{lunge}|X+B|-|X|=\kappa_{9(r+3)}(B)\leq |B|+r\;\und\;|X+B|\leq |X|+|B|+r.\ee
In view of \eqref{p-bigwithr} and $r\geq 1$, we can apply Corollary \ref{cor-atombounds}.3 to conclude that \be\label{X-size}|X|=\alpha_{9(r+3)}(B)\leq 10r+26.5+\sqrt{2r+\frac{17}{4}}\leq \frac{1963}{9253}p,\ee with the latter inequality holding by hypothesis.
Since $B$ is $9(r+3)$-separable with $X$ a $9(r+3)$-atom of $B$, it follows by definition of an atom that $|X+B|\leq p-9(r+3)$. Thus, since $|B|\geq 9(r+3)$, it follows that $X$ is also $9(r+3)$-separable with $\kappa_{9(r+3)}(X)\leq |X+B|-|B|\leq |X|+r$.  Let $Y$ be a $9(r+3)$-atom of $X$. Then \be\label{X+Y}|X+Y|-|Y|\leq \kappa_{9(r+3)}(X)\leq |X|+r\;\und\;|X+Y|\leq |X|+|Y|+r.\ee Since $X$ and $Y$ are $9(r+3)$-atoms, we have \be\label{startitup}|X|,\,|Y|\geq 9(r+3)\;\und\; |X+Y|\leq p-9(r+3),\ee and Corollary \ref{cor-atombounds}.3  again gives  \be\label{Y-size}|Y|=\alpha_{9(r+3)}(X)\leq 10r+26.5+\sqrt{2r+\frac{17}{4}}\leq \frac{1963}{9253}p,\ee with the latter inequality holding by hypothesis.

In view of \eqref{X+Y} and \eqref{startitup}, we have
\be\label{X+Y-hyp}|X+Y|\leq |X|+|Y|+r\leq \min\Big\{|X|+|Y|+\frac19\min\{|X|,\,|Y|\}-3,\; p-9(r+3)\Big\}.\ee In view of \eqref{X-size}, \eqref{Y-size} and \eqref{startitup}, we have
\be\label{beta-ratio} \frac{\min\{|X|,|Y|\}}{\max\{|X|,|Y|\}}\geq \frac{9r+27}{10r+26.5+\sqrt{2r+\frac{17}{4}}}\geq 0.8877,\ee with the latter inequality following by  a routine calculus minimization question showing that the above expression is minimized for $r=24$.
Now \eqref{X-size}, \eqref{Y-size},  \eqref{X+Y-hyp} and \eqref{beta-ratio} allow us to  apply Corollary \ref{cor-fouier-II-distinct} to $X+Y$ to conclude that \be\label{digup}\ell_d(X)\leq |X|+r+1\;\und\;\ell_d(Y)\leq |Y|+r+1\ee for some nonzero $d\in G$.
We have  $|B|,|X|\geq 9(r+3)\geq 3r+6$ and $|X+B|\leq p-9(r+3)\leq p-(3r+6)$, so \eqref{lunge} and  \eqref{digup} allow us to  apply Theorem \ref{thm-mario-apred}.1 (with $h=r+1$) to the additive trio $(X,B,-(X+B)^\mathsf c)$ to conclude that $$\ell_d(B)\leq |B|+r+1.$$ We again have $|A|,|B|,|C|\geq 9(r+3)\geq 3r+6$, so this allows us to apply Theorem \ref{thm-mario-apred}.1 (with $h=r+1$) to the additive trio $(B,A,C)$ to conclude that $\ell_d(A)\leq |A|+r+1$, $\ell_d(B)\leq |B|+r+1$ and $\ell_d(C)\leq |C|+r+1$, which is the desired conclusion of the theorem, completing the proof.
\end{proof}

\begin{lemma}
\label{lem-sqrt-shortcalc}
Let $r\geq -1$ be an integer. Then $\sqrt{2r+\frac{17}{4}}\leq \frac{10}{3}+0.16797\,r$.
\end{lemma}

\begin{proof}The desired inequality is true for $r\in \{-1,0\}$, so we can assume $r\geq 1$ and
seek to maximize the expression $\frac{\sqrt{2r+\frac{17}{4}}-\frac{10}{3}}{r}$ for $r\geq 1$ an integer. The derivative of this expression is $\frac{-51-12r+20\sqrt{17+8r}}{6r^2\sqrt{17+8r}}$, which is positive for $r=1$ and negative for sufficiently large $r$ having a unique root (larger than $1$) at $r\approx 15.5923$, with  the expression then maximized when $r=16$. Thus $\frac{\sqrt{2r+\frac{17}{4}}-\frac{10}{3}}{r}\leq \frac{\sqrt{32+\frac{17}{4}}-\frac{10}{3}}{16}\leq 0.16797$, as desired.
\end{proof}

Removing the unwanted hypothesis $\frac{9253}{1963}\Big(10r+26.5+\sqrt{2r+\frac{17}{4}}\Big)\leq p$ from Theorem \ref{thm-r-O(p)-a}  is simply a matter of imposing a sufficiently stronger small sumset hypothesis that implicitly forces this needed hypothesis. Thus we now proceed with the proof of Theorem \ref{thm-p-O(r)}.

\begin{proof}[Proof of Theorem \ref{thm-p-O(r)}]
If we can show $\frac{9253}{1963}\Big(10r+26.5+\sqrt{2r+\frac{17}{4}}\Big)\leq p$ is implied by the hypotheses of Theorem \ref{thm-p-O(r)}, then we can apply Theorem \ref{thm-r-O(p)-a} to $A+B$ to complete the proof. Thus, in view of Lemma \ref{lem-sqrt-shortcalc}, it suffices to show \be\label{goalto}\frac{9253}{1963}\Big(10.16797\,r+26.5+\frac{10}{3}\Big)\leq p.\ee
By hypothesis of Theorem \ref{thm-p-O(r)}, we have $2(r+3)\alpha^{-1}+r\leq |A|+|B|+r=|A+B|\leq p-9(r+3)$, where $\alpha=0.0527$, implying \be\label{hotub}(2\alpha^{-1}+10)r+(27+6\alpha^{-1})\leq p.\ee
Since $(2\alpha^{-1}+10)\geq \frac{9253}{1963}\cdot 10.16797$ and $(27+6\alpha^{-1})\geq \frac{9253}{1963}\Big(26.5+\frac{10}{3}\Big)$, \eqref{hotub} implies \eqref{goalto}, as needed.
\end{proof}

Theorem \ref{thm-p-O(r)} is not the only variation derivable from Theorem \ref{thm-r-O(p)-a}. For instance, the hypothesis $\frac{9253}{1963}\Big(10r+26.5+\sqrt{2r+\frac{17}{4}}\Big)\leq p$ can also be removed by increasing the density restriction.

\begin{corollary}\label{cor1}
Let $G=\Z/p\Z$ with $p\geq 2$ a prime, let $A,\, B\subseteq G$ be nonempty subsets with $A+B\neq G$,  and set $C=-\,G\setminus (A+B)$. Suppose
  $|A|\geq |B|$   and  \begin{align*}&|A+B|=|A|+|B|+r\leq \min\{|A|+\frac19|B|-3,\quad p-29(r+3)\}.\end{align*}
Then there are arithmetic progressions $P_A,\,P_B,\,P_{C}\subseteq G$ all having  common difference  such that
\begin{align*} &X\subseteq P_X\;\und\; |P_X|\leq |X|+r+1\quad \mbox{ for all $X\in \{A,B,C\}$}.\end{align*}
\end{corollary}

\begin{proof}
As in the proof Theorem \ref{thm-p-O(r)}, it suffices to show \eqref{goalto} holds as then Theorem \ref{thm-r-O(p)-a} can be applied to $A+B$ to complete the proof. By hypothesis, we have $18(r+3)+r\leq |A|+|B|+r=|A+B|\leq p-29(r+3)$, implying $48r+141\leq p$.
Since $48\geq \frac{9253}{1963}\cdot 10.16797$ and $141\geq \frac{9253}{1963}\Big(26.5+\frac{10}{3}\Big)$, this implies  \eqref{goalto}, as needed.
\end{proof}

We would like to use the improved mid-range density result from \cite{lev-3k-4-highenrg} in the proof of Theorem \ref{thm-ideal-density}. Unfortunately, as stated, it includes the hypothesis $|A|\geq 101$ in order to yield a slightly better small doubling constraint $\alpha<0.59|B|-3$. This makes their result unusable for our purposes since we will need to apply it to an auxiliary set $X$ that may be much smaller than $A$. However, it is a simple matter to vary the calculations at the end of their proof to obtain the following flexible version of Lev and Shkredov's result valid for smaller sets.  While the proof only requires $K\leq 3$, we need $c>0$ for the result to be non-void, which requires $K\lesssim 2.595$. Additionally, by using Proposition \ref{lemma-four-case1} and Lemma \ref{lem-freiman-circle} instead of   \cite[Lemma 2]{lev-rect-threshold}, we extend the arguments of Lev and Shkredov to higher densities $|A|\leq \frac{3}{4K}p$ rather than $|A|<\frac{1}{12}p$. Note, the methods of \cite{pablo-grynk-3k-4-A=B} (based on the original fourier sum estimates of Freiman and Rodseth) still yield better constants at high densities (roughly once $K\leq 2.349$ and $|A+A|\geq 0.364\,p$, for $s=7$), so the improvements here are of principal interest for the   higher mid-range densities we will encounter in the proof of Theorem \ref{thm-ideal-density}.

\begin{theorem}\label{thm-lev-sh-fixed}
Let $G=\Z/p\Z$ with $p\geq 2$ a prime, let $A\subseteq G$ be a nonempty subset with $2A\neq G$,  and set $C=-\,G\setminus (2A)$.
Let  $K\in (2,3)$ be a  real number, let $s\geq K^2$ and let $$c=\frac{-27K+9K^2+s(9+9K-9K^2+12K^3-4K^4)}{4s(3-K)K^4}.$$
Suppose  $s\leq |A|\leq  c\,p$,\; $|2A|=2|A|+r\leq \frac34\,p$ and $|2A|<K|A|-3$.
Then there are arithmetic progressions $P_A,\,P_{C}\subseteq G$ both having  common difference  such that $X\subseteq P_X$ and $|P_X|\leq |X|+r+1$ for all $X\in \{A,C\}$.
\end{theorem}

\begin{proof}
  The proof is nearly identical to \cite[Theorem 4]{lev-3k-4-highenrg} apart from the usage of Proposition \ref{lemma-four-case1} and Lemma \ref{lem-freiman-circle} in place of \cite[Corollary 1]{lev-3k-4-highenrg}.
  We sketch the minor differences. Viewing $cK$ as a function of $s$ and $K$, it has positive derivative with respect to $s$ with $cK\rightarrow \frac{9+9K-9K^2+12K^3-4K^4}{4(3-K)K^4}<1$ (as $2<K<3$). This show that $1-cK>0$ always holds.
Set  $K':=\frac{|2A|}{|A|}<K$. While their notation is different from ours (their $\hat A(\chi)$ for $\chi \in \hat{\mathbb F}_p$ are the same as our  $S_A(x)$ for $x\in [0,p-1]$), defining $\eta$ by $\max\{|S_A(x)|: x\in [1,p-1]\}=\eta|A|$ agrees with how $\eta$ is defined in \cite[Proof Theorem 4]{lev-3k-4-highenrg}. Since the derivative (with respect to $K'$) of $\frac{1}{K'}+\frac{K'}{|A|}$ is $\frac{1}{|A|}-\frac{1}{(K')^2}<\frac{1}{|A|}-\frac{1}{K^2}\leq  0$ (in view of $|A|\geq s\geq K^2$), the fourier calculation in \cite[Proof Theorem 4]{lev-3k-4-highenrg} ending with the estimate $\geq \Big(1+\frac{1}{K'}-\frac{3-K'}{|A|}\Big)|A|^3$ is made smaller by replacing $K'$ with its upper bound $K$, allowing us to assume $K=K'$ for the numerical calculations from that point onward. The inequality
  \be\label{tacky}\eta^2\geq \frac{1}{K(1-c K)}\Big(1+\frac{1}{K}-\frac{3-K}{|A|}-c K^2\Big)\geq \frac{1}{K(1-c K)}\Big(1+\frac{1}{K}-\frac{3-K}{s}-c K^2\Big)\ee follows by the authors' arguments, with the latter inequality holding since $K\leq 3$ allows us to apply the hypothesis $|A|\geq s$.

  Let $A'\subseteq A$ be a maximal cardinality subset such that $A'+A'$ rectifies. Let $n=|A'|$. We have $|2A|\leq \frac34\,p$ and  $2|A|+r=|2A|<K|A|-3$ by hypothesis, whence $r<(K-2)|A|-3$. Thus $\frac34\,p\geq |2A|=2|A|+r> \frac{2(r+3)}{K-2}+r\geq 3r+6$, implying $4r+9\leq p$. This gives us the hypotheses of Proposition \ref{lemma-four-case1}. As a result, if
$3n-4\geq |2A|$, then Proposition \ref{lemma-four-case1} yields the desired conclusion. Therefore we can instead assume $n\leq \frac13(|2A|+3)<\frac{1}{3}K|A|$. This ensures, for any $x\in [1,p-1]$, that less than   $\frac{1}{3}K|A|$ of the terms in the fourier sum $S_A(x)$ lies in an open half arc. Otherwise the inverse of $x$ modulo $p$ will be the difference of an arithmetic progression of length at most $\frac{p+1}{2}$ containing at least $\frac{1}{3}K|A|$ elements of $A$ (as argued in Claim A of the proof of Theorem \ref{fouier-II-distinct}), contrary to what was just concluded. Applying Lemma \ref{lem-freiman-circle}, we conclude that $\eta|A|<\frac23K|A|-|A|$, whence $\eta<\frac23K-1$.
Applying this estimate  in \eqref{tacky} yields a linear inequality in $c$ that (in view of $1-cK>0$) implies   $c>\frac{-27K+9K^2+s(9+9K-9K^2+12K^3-4K^4)}{4s(3-K)K^4}$ after rearranging all expressions, which is contrary to hypothesis.
\end{proof}

\section{Optimal Density}

The final section is devoted to the proof of Theorem \ref{thm-ideal-density}. We take the  general approach of  \cite{serra-3k-4-plunnecke} as a broad outline, though we modify and add to the arguments and incorporate several new advances, including our Theorem \ref{thm-p-O(r)}, the reduction result of Huichochea \cite{huicochea-reduction-3k-4}, the mid-range density result of Lev and Shkredov \cite{lev-3k-4-highenrg} as modified in Theorem \ref{thm-lev-sh-fixed}, and  the improved bounds for $2$-atoms given in \cite{Grynk-book} \cite{hyperatoms}. The use of   \cite[Lemma 4.5]{serra-3k-4-plunnecke} is replaced by an alternative reduction argument utilizing Theorem \ref{thm-mario-apred}.



\begin{proof}[Proof of Theorem \ref{thm-ideal-density}]
Since $A+B\neq G$, we have $r\geq -1$ by the Cauchy-Davenport Theorem (Theorem \ref{thm-cdt}).
Let $$\gamma=0.01$$
In view of the hypotheses, we have
\be\label{ABC-size} |A|\geq |B|\geq \gamma^{-1}(r+3)=100(r+3)\quad\und\quad |C|\geq r+3.\ee
If $r\leq 0$, then (as mentioned in the Small $r$ Subsection of the Introduction), Conjecture \ref{conj-crit-pair} is known yielding the desired result. If $r=1$, then \eqref{ABC-size} ensures $|A|\geq |B|\geq 6$ and $|C|\geq 4$, in which case Conjecture \ref{conj-crit-pair} holds for $A+C$ (as mentioned in the Small $r$ Subsection of the Introduction), yielding the desired result for $A+B$ as well (via Proposition \ref{lemma-dual}). Therefore we can assume $$r\geq 2.$$
Our hypotheses ensure
$2\gamma^{-1}(r+3)+r\leq |A|+|B|+r=|A+B| \leq p-r-3$, whence
\be\label{r-p-hyp}p\geq 2(\gamma^{-1}+1)r+6\gamma^{-1}+3\geq 202r+603\geq 1007,\ee with the final inequality in view of $r\geq 2$.

Since $|A|,|B|,|C|\geq r+3$, it follows that $A$ is $(r+3)$-separable with $\kappa_{r+3}(A)\leq |A+B|-|B|\leq |A|+r$. Let $X\subseteq G$ be an $(r+3)$-atom for $A$. Then $|X|\geq r+3$, $|X+A|\leq p-r-3$ and $|A+X|-|X|=\kappa_{r+3}(A)\leq |A|+r$, implying
\be\label{AX-1} |A+X|\leq |A|+|X|+r.\ee
Since $r\geq 1$, \eqref{r-p-hyp} allows us to apply Corollary \ref{cor-atombounds} yielding
\be\label{atombounds-init}r+3\leq |X|\leq\Big\lfloor 2r+2.5+\sqrt{2r+\frac{17}{4}}\Big\rfloor.\ee
By calculating its derivative, the expression $\frac{3r+2.5+\sqrt{2r+\frac{17}{4}}}{r+3}$ is maximized when $r=84$ via basic calculus. If $r\in [2,79]\cup [89,\infty)$, then this expression is maximized at $r=89$, which combined with \eqref{atombounds-init} yields \be\label{X+r+3-upper}|X|+r\leq \frac{283}{92}(r+3).\ee Taking advantage of the fact that $|X|$ is an integer, it can be manually checked that \eqref{atombounds-init} yields \eqref{X+r+3-upper} for $r\in [80,89]$ as well, meaning \eqref{X+r+3-upper} holds in general. Similar analysis (checking the values $r\in [19,25]$ manually for the first case below), shows that
\be\label{X+r+3-upper-alt}|X|+r\leq 3.143\,(r+2),\quad |X|+r\leq 3.1\,(r+2.5),\quad |X|+r\leq 3.0885\,(r+2.7) .\ee


Using \eqref{AX-1} and the Ruzsa-Pl\"unnecke Bounds (Theorem \ref{thm-plunnecke-petridis}), we find a nonempty subset $A'\subseteq A$ such that
\be\label{A'-nX-pluneck}|A'+nX|\leq \Big(\frac{|A+X|}{|A|}\Big)^n|A'|\leq \Big(1+\frac{|X|+r}{|A|}\Big)^n|A'|\quad\mbox{ for all $n\geq 1$}.\ee
In particular, letting $$1-\epsilon =|A'|/|A|,$$
we have $\epsilon\in [0,1)$ and
\be\label{plunnecke-n=1} |A'+X|\leq |A'|+(1-\epsilon)(|X|+r)=|A'|+|X|+r-\epsilon(|X|+r).\ee
Since $|A'+X|\leq |A+X|\leq p-r-3<p$, the Cauchy-Davenport Theorem (Theorem \ref{thm-cdt}) ensures $|A'+X|\geq |A'|+|X|-1$, which combined with \eqref{plunnecke-n=1} and $2\leq r\leq |X|-3$ forces $-1\leq r-\epsilon(|X|+r)\leq |X|-3-\epsilon(2|X|-3)$, in turn implying $\epsilon\leq \frac{|X|-2}{2|X|-3}<\frac12$.
Thus (from \eqref{ABC-size}) \be\label{epsilon-ratio}0\leq \epsilon<\frac12\;\und\;|A'|>\frac12|A|\geq \frac12\gamma^{-1}(r+3)=50r+150.\ee

In view of \eqref{ABC-size} and \eqref{X+r+3-upper}, we have
\be\label{XroverA}\frac{|X|+r}{|A|}\leq \frac{|X|+r}{\gamma^{-1}(r+3)}\leq  \frac{283}{92}\gamma\leq 0.0308.\ee

We claim that  \be\label{stillness}(1+z)^n\leq 1+nz+0.466\,n^2z^2=1+nz(1+0.466\, nz)\ee for any non-negative real number $z\leq 0.0313$ and any positive integer $n\leq 8$. Indeed, for $n=2$, we have $1+nz+0.466n^2z^2-(1+z)^n=0.864z^2\geq 0$, while for $n\geq 3$ the smallest positive root of $1+nz+0.466n^2z^2-(1+z)^n$, for $n\in [3,8]$, is always greater than $0.0313$ with this expression positive for $0<z\leq 0.0313$, establishing \eqref{stillness}.

Now  \eqref{XroverA} allows us to apply \eqref{stillness} with $z=\frac{|X|+r}{|A|}$, which combined with  \eqref{A'-nX-pluneck}, \eqref{XroverA} and $\gamma=0.01$ yields
\begin{align}\nn|A'+nX|&\leq |A'|+(|A'|/|A|)n(|X|+r)\Big(1+0.466\cdot n\big(\frac{|X|+r}{|A|}\big)\Big)\\
&\leq |A'|+(|A'|/|A|)n(|X|+r)\Big(1+0.01434\cdot n\Big)\quad\mbox{ for all $n\in [2,8]$}.
\label{A'-nX-simpl}
\end{align}
By hypothesis, $|A|+\gamma^{-1}(r+3)+r\leq |A|+|B|+r=|A+B|\leq p-r-3$, implying
\be\label{p-Abound}p-|A|\geq (\gamma^{-1}+1)(r+3)+r=102r+303.\ee
Using the estimates $|A'|\leq |A|$, \eqref{X+r+3-upper} and \eqref{p-Abound} in \eqref{A'-nX-simpl}, we find that \begin{align}\nn|A'+8X|&\leq |A|+8\frac{283}{92}(r+3) (1+0.01434\cdot 8)\\&\leq |A|+27.432(r+3)\leq p-102r-303+27.432(r+3)\leq p-5.\nn\end{align}
This allows us to apply the Cauchy-Davenport Theorem (Theorem \ref{thm-cdt}) to $A'+8X$ (or any sub-sumset later in the proof). Doing so,  we see that \eqref{A'-nX-simpl} implies
\be\label{nX-simple} |nX|\leq (|A'|/|A|)n(|X|+r)\Big(1+0.01434\cdot n\Big)+1\quad\mbox{ for all $n\in [2,8]$}.\ee

By calculating its derivative, basic calculus shows that  $26.5-0.246719\,r'+\sqrt{2r'+\frac{17}{4}} $ is maximized (over all integers $r'\geq -1$) for $r'=6$, whence $26.5+\sqrt{2r'+\frac{17}{4}}\leq 0.246719\,r'+29.05082$. Hence
\be\label{r'-hyp}\frac{9253}{1963}\Big(10\,r'+26.5+\sqrt{2r'+\frac{17}{4}}\Big)\leq 48.3\,r'+136.937\quad\mbox{ for all integers $r'\geq -1$.}\ee

\textbf{Claim A:}  If $\ell_d(A')\leq |A'|+48.9\,r+145.3$  for some nonzero $d\in G$, then all desired conclusions in Theorem \ref{thm-ideal-density} follow.

\begin{proof}[Proof of Claim A]
Let $|A'+X|=|A'|+|X|+r'$, let $h=48.9\,r+145.3$, and let $C'=-(A'+X)^\mathsf c$.  By \eqref{plunnecke-n=1}, we have $r'\leq r$.  By \eqref{epsilon-ratio},  we have
$|A'|> 50r+150\geq r+4+h$. By \eqref{atombounds-init}, we have $|X|\geq r+3$. By \eqref{AX-1}, \eqref{X+r+3-upper} and \eqref{p-Abound}, we have \be\label{chainit}|A'+X|\leq |A+X|\leq |A|+|X|+r\leq p-102\,r-303+\frac{283}{92}(r+3),\ee whence $|C'|\geq 98.9\,r+293.7\geq 2h+r+3$. Thus we can apply Theorem \ref{thm-mario-apred}.1 to the additive trio $(A',C',X)$ (with the $h$ in Theorem \ref{thm-mario-apred} taken to be $\lfloor h\rfloor$) to conclude that $\ell_d(X)\leq |X|+r+1$.

Let $C''=-(A+X)^\mathsf c$.  By \eqref{chainit}, We have $|A+X|\leq |A|+|X|+r\leq p-98.9\,r-293.7<p-(r+3)$, whence $|C''|\geq r+4$.  By \eqref{ABC-size}, \eqref{X+r+3-upper} and \eqref{atombounds-init},  we have $|A|\geq 100(r+3)>\frac{283}{92}(r+3)\geq |X|\geq r+3$. By \eqref{r-p-hyp}, we have $p\geq 35r+45$. Thus we can apply  Theorem \ref{thm-mario-apred}.2 to the additive trio $(X,A,C'')$ (with $h=r+1$) to conclude that $\ell_d(A)\leq |A|+r+1$.

By \eqref{ABC-size}, we have $|A|\geq |B|\geq 100(r+3)\geq 3r+6$ and $|C|=p-|A+B|\geq r+3$. Thus we can apply  Theorem \ref{thm-mario-apred}.1 (with $h=r+1$) to the additive trio $(A,B,C)$ to yield all desired conclusions for Theorem \ref{thm-ideal-density}.
\end{proof}

\textbf{Claim B:}  If $1\leq n\leq 8$ and $\ell_d(nX)\leq 2|nX|-2$ for some nonzero $d\in G$, then $\ell_d(2X)< \frac{p}{2}$.

\begin{proof}[Proof of Claim B]
By \eqref{nX-simple} and   \eqref{X+r+3-upper}, we have $$|nX|\leq |8X|\leq 27.44(r+3)+1=27.44\,r+83.32.$$ Thus, if $n\geq 2$, then  the hypothesis of Claim B together with  \eqref{r-p-hyp} yields $\ell_d(2X)\leq \ell_d(nX)\leq 2|nX|-2\leq  54.88\,r+164.64<\frac{p}{2}$, as desired. If $n=1$, then  the hypothesis of Claim B together with  \eqref{X+r+3-upper} and \eqref{r-p-hyp} yields $\ell_d(2X)< 2\ell_d(X)\leq 4|X|-4\leq  8.31\,r+33<\frac{p}{2}$, as desired.
\end{proof}

\textbf{Claim C:}
\begin{align*}
&|2X|< (0.9)(|A'|/|A|)\big(2.0574(|X|+r)+3\big)\\ &|4X|< (0.9)(|A'|/|A|)\big(4.2295(|X|+r)+3\big),\quad\und \\
&|8X|< (0.9)(|A'|/|A|)\big(8.9178(|X|+r)+3\big).\end{align*}

\begin{proof}
 By \eqref{nX-simple}, we have  \be\label{8-upit-a} |8X|\leq (|A'|/|A|)8.9178(|X|+r)+1.\ee
In view of \eqref{8-upit-a}, \eqref{X+r+3-upper}, and \eqref{epsilon-ratio}, it follows that
\begin{align} |8X|\leq 27.44(r+3)+1<|A'|.\label{8-prep}
\end{align}
Let $r'$ be defined by $|A'+8X|=|A'|+|8X|+r'$.

Suppose $$r'\leq \frac19|8X|-3.$$ Then \eqref{8-prep} implies \be\label{r1-b-8}r'\leq 3.05\,r+6.26,\ee and then \eqref{r'-hyp} implies $$\frac{9253}{1963}\Big(10\,r'+26.5+\sqrt{2r'+\frac{17}{4}}\Big)\leq 48.3\,r'+136.937\leq 147.4\,r+439.3<p,$$ with the final inequality by \eqref{r-p-hyp}.
By   \eqref{8-prep},  \eqref{p-Abound} and \eqref{r1-b-8}, we have $|A'+8X|=|A'|+|8X|+r'\leq |A|+27.44(r+3)+1+r'\leq
p-102\,r-303+27.44(r+3)+1+r'=
p-74.56\,r-219.68+r'\leq
p-23.44\,r'-66.64
\leq p-9(r'+3)$.
 Thus we can apply Theorem \ref{thm-r-O(p)-a} to $A'+8X$ to conclude via \eqref{r1-b-8} that $\ell_d(A')\leq |A'|+r'+1\leq |A'|+3.05\,r+7.26$ for some nonzero $d\in G$, and then Claim A completes the proof.   So we instead conclude that $r'>\frac19|8X|-3$, meaning $|A'+8X|> |A'|+\frac{10}{9}|8X|-3$. Combined with \eqref{A'-nX-simpl}, we now have $|8X|< (0.9)(|A'|/|A|)\big(8.9178(|X|+r)+3\big)$, as desired.

 We have $|4X|\leq (|A'|/|A|)4.2295(|X|+r)+1$ and $|2X|\leq (|A'|/|A|)2.0574(|X|+r)+1$ by \eqref{nX-simple}. Thus, since $|2X|\leq |4X|\leq |8X|$, the above arguments can also be applied to $2X$ and $4X$ to conclude that $|A'+4X|>|A'|+\frac{10}{9}|4X|-3$ and $|A'+2X|>|A'|+\frac{10}{9}|2X|-3$, yielding the stated bounds for $|4X|$ and $|2X|$.
\end{proof}

\textbf{Claim D:}  $\ell_d(2X)< \frac{p}{2}$ for some nonzero $d\in G$.

\begin{proof}[Proof of Claim D]
Assume by contradiction that
\be\label{ell-up}\ell_d(2X)\geq \frac{p}{2}\quad \mbox{ for all nonzero $d\in G$}. \ee By Claim C, \eqref{X+r+3-upper}, \eqref{X+r+3-upper-alt} and \eqref{r-p-hyp}, we have
\begin{align}
&|X|\leq 3.143(r+2)-r\leq 2.143\,r+6.286<0.01061\,p, \nn\\
&|2X|\leq (0.9)(2.0574)(3.1)(r+2.5)+2.7\leq 5.741\,r+17.051<0.02843\,p,\nn\\
&|4X|\leq (0.9)(4.2295)(3.0885)(r+2.7)+2.7\leq  11.757\,r+34.443<0.05821\,p,\nn\\
&|8X|\leq (0.9)(8.9178)\frac{283}{92}(r+3)+2.7\leq 24.689\,r+76.767<0.128\,p.\label{4-kings}
\end{align}
We handle three cases.

\textbf{CASE D1:} $|X|\geq 11$ and $|X|\geq r+4$.

In this case, \eqref{4-kings} allows us to apply Theorem \ref{thm-lev-sh-fixed} with $s=11$ to $X+X$, $2X+2X$ and $4X+4X$ using the values $K=2.572 $ and $c> 0.0111$, $K= 2.552$ and $c> 0.0289$, and
$K= 2.515$ and $c> 0.0588 $, respectively. Thus, in view of Claim B and \eqref{ell-up}, it follows that $|2X|\geq 2.572|X|-3$, \, $|4X|\geq 2.552|2X|-3$ and $|8X|\geq 2.515|4X|-3$. Iterating these bounds and comparing with the bound from Claim C combined with the  case hypothesis $r\leq |X|-4$, we find
$$16.5|X|-29.8<|8X|\leq 16.06|X|-29.4,$$ which contradicts that $|X|\geq 2$.

\textbf{CASE D2:} $|X|\leq 10$.

In this case, we have $|X|\leq \lceil \log_2 p\rceil $ by \eqref{r-p-hyp}. Thus we can apply Theorem \ref{thm-smallr} to $X+X$, which in view of Claim B and \eqref{ell-up} implies $|2X|\geq 3|X|-3$. We have  $|2X|\geq 3|X|-3\geq 3(r+3)-3\geq 12$ (by \eqref{atombounds-init}). Thus,  as in CASE D1, we can assume $|4X|\geq 2.552|2X|-3$ and $|8X|\geq 2.515|4X|-3$. Iterating these  bounds and comparing with the bound from Claim C combined with  the estimate  $r\leq |X|-3$ (from \eqref{atombounds-init}), we find
$$19.254|X|-29.8<|8X|\leq 16.06|X|-21.378,$$ which implies $|X|\leq 2$, contradicting that $|X|\geq r+3\geq 5$ (by \eqref{atombounds-init}).

\textbf{CASE D3:} $|X|=r+3$ and $r\geq 8$.

In this case, by Claim C and \eqref{r-p-hyp}, we have
\begin{align}
&|X|=r+3<0.00498\,p \nn\\
&|2X|\leq 3.704\,r+8.255<0.0184\,p\nn\\
&|4X|\leq 7.614\,r+14.12<0.0377\,p\nn\\
&|8X|\leq 16.053\,r+26.78<0.0795\,p.\nn
\end{align}
This allows us to apply Theorem \ref{thm-lev-sh-fixed} with $s=11$ to $X+X$, $2X+2X$ and $4X+4X$ using the values $K=2.578 $ and $c> 0.00561$, $K= 2.564$ and $c> 0.01845$, and
$K= 2.541$ and $c> 0.0382 $, respectively. Thus, in view of Claim B and \eqref{ell-up}, it follows that $|2X|\geq 2.578|X|-3$, \, $|4X|\geq 2.564|2X|-3$ and $|8X|\geq 2.541|4X|-3$. Iterating these bounds and comparing with the bound from Claim C combined with the case hypothesis $r=|X|-3$, we find
\be\label{laxup}16.795|X|-30.169<|8X|\leq 16.06|X|-21.378,\ee which in view of  $|X|=r+3\geq 11$ is only possible if $r=8$ and $|X|=11$. For this case, we have $|2X|\geq \lceil 2.578|X|-3\rceil=26$, \ $|4X|\geq \lceil 2.564|2X|-3\rceil \geq 64$ and $|8X|\geq \lceil 2.541|4X|-3\rceil\geq 160$, while \eqref{laxup} implies $|8X|\leq 155$, contradicting this to complete Claim D.
\end{proof}

In view of Claim D, we have $\ell_{d}(X)\leq \ell_{d}(2X)<\frac{p}{2}$ for some nonzero $d\in G$, ensuring that both $X+X$ and $2X+2X$ rectify. This will allow us to apply Theorem \ref{thm-3k4Z} to $X+X$ and $2X+2X$ as done when proving Theorem \ref{fouier-II-distinct}.

\textbf{Claim E:} $|2X|\geq 3|X|-3$.

\begin{proof}[Proof of Claim E] Let $$h=|X|-2$$ and let $C'=-(A+X)^\mathsf c$.
 Note that \be\label{r-h-p}r\leq |X|-3=h-1\ee by \eqref{atombounds-init}.
Assume by contradiction that $|2X|\leq 3|X|-4$. Then,
since $X+X$ rectifies, Theorem \ref{thm-3k4Z} implies that $$\ell_d(X)\leq |X|+(|X|-3)= |X|+h-1\leq |X|+h,$$ for some nonzero $d\in G$.  By \eqref{AX-1} and \eqref{r-h-p}, we have $|A+X|\leq |A|+|X|+h-1$ and $|X|=h+2$. By \eqref{ABC-size} and \eqref{X+r+3-upper}, we have $|A|\geq  100(r+3)\geq 100\cdot\frac{92}{283}(|X|+r)=100\cdot \frac{92}{283}(h+2+r)\geq h+4$.
By \eqref{AX-1}, \eqref{p-Abound} and  \eqref{X+r+3-upper}, we have $|A+X|+h+3\leq |A|+2|X|+r+1\leq |A|+2(|X|+r)\leq  p-102\,r-303+\frac{283}{46}(r+3)< p$, ensuring that $|C'|\geq h+3$. By \eqref{X+r+3-upper} and  \eqref{r-p-hyp}, we have $35h+10=35|X|-60\leq 35\cdot \frac{283}{92}(r+3)-60\leq 108r+263<p$. Thus we can apply Theorem \ref{thm-mario-apred}.2 to the additive trio $(X,A,C')$ (with $h=|X|-2$) to conclude that $$\ell_d(A)\leq |A|+h=|A|+|X|-2.$$

By hypothesis, we have $|A+B|= |A|+|B|+r$. By \eqref{ABC-size} and \eqref{X+r+3-upper}, we have $|A|,\,|B|\geq  100(r+3)\geq \frac{283}{46}(r+3)> r+2|X|=r+4+2h$. By \eqref{ABC-size}, we have $|C|\geq r+3$. This allows us to apply Theorem \ref{thm-mario-apred}.1 to the additive trio $(A,B,C)$ to yield all desired conclusions for Theorem \ref{thm-ideal-density}. So we can now  instead assume $|2X|\geq 3|X|-3$, as claimed.
\end{proof}

Let $C'=-(A'+2X)^\mathsf c$.
If $|4X|\geq 3|2X|-3$, then Claims E and C together with $r\leq |X|-3$ imply that $9|X|-12\leq |4X|\leq (0.9)(4.2295(2|X|-3)+3)\leq 7.6131|X|-8.719$, implying $|X|\leq 2$, which contradicts that $|X|\geq r+3=5$ by \eqref{atombounds-init}. Therefore we must instead have
\be\label{4X} |4X|\leq 3|2X|-4.\ee
As a result, since $2X+2X$ rectifies, we can apply Theorem \ref{thm-3k4Z} to $2X+2X$ to conclude via Claim  C and \eqref{X+r+3-upper} that \be\label{tagyourup}\ell_d(2X)\leq |2X|+h\quad\mbox{ where $h:=|2X|-3\leq 5.696\,r+16.788$.}\ee
If $|A'+2X|\geq |A'|+|2X|+h$, then combining this with  Claim E, \eqref{A'-nX-simpl} and $r\leq |X|-3$ (by \eqref{atombounds-init}) yields  $|A'|+6|X|-9\leq |A'|+|2X|+h\leq |A'+2X|\leq |A'|+4.115\,|X|-6.172$, implying $|X|\leq 1$, which contradicts that $|X|\geq r+3\geq 5$ by \eqref{atombounds-init}.
Therefore we instead conclude that $|A'+2X|\leq |A'|+|2X|+h-1$. By definition of $h$, we have $|2X|=h+3$. By \eqref{epsilon-ratio} and \eqref{tagyourup}, we have $|A'|>50\,r+150\geq 5.696\,r+20.788\geq h+4$. By  \eqref{A'-nX-simpl},  \eqref{X+r+3-upper}, \eqref{tagyourup} and  \eqref{p-Abound}, we have
 $|A'+2X|+h+4\leq |A'|+6.329\,r +22.986+h\leq |A|+12.025\,r+39.774\leq p-89.975\,r-263.226<p$, ensuring $|C'|\geq h+4$. Finally, in view of \eqref{tagyourup} and \eqref{r-p-hyp}, we have $35h+10\leq 199.36\,r+597.58<p$. Thus we can apply Theorem  \ref{thm-mario-apred}.2 to the additive trio $(2X,A',C')$ to conclude that $\ell_d(A')\leq |A'|+h\leq |A'|+5.696\,r+16.788$, with the latter inequality by \eqref{tagyourup}. This allows us to  apply Claim A to complete the proof.
\end{proof}

\end{document}